\documentclass[12pt]{amsart}
\pdfoutput=1
\usepackage[a4paper,margin=2.5cm,headsep=20pt]{geometry}
\usepackage[hidelinks,pdfstartview=FitH]{hyperref}
\usepackage{amssymb,enumitem,tikz}
\usepackage[T1]{fontenc}
\usepackage[english]{babel}
\usepackage[utf8]{inputenc}
\setitemize{itemsep=3pt,topsep=2pt,leftmargin=2.5em}
\setenumerate{label=(\roman*),itemsep=3pt,topsep=2pt,leftmargin=2.5em}
\usetikzlibrary{arrows,automata,arrows.meta,calc}
\tikzset{>=To, baseline=(current bounding box.center),inner sep=3pt}

\numberwithin{equation}{section}
\newtheorem{thm}{Theorem}[section]
\newtheorem{prop}[thm]{Proposition}
\newtheorem{cor}[thm]{Corollary}
\newtheorem{lemma}[thm]{Lemma}
\theoremstyle{remark}
\newtheorem{rmk}[thm]{Remark}

\theoremstyle{definition}

\def\sw#1{{\sb{(#1)}}}
\newcommand{\co}{\operatorname{co}}
\newcommand{\id}{\mathrm{id}}
\newcommand{\C}{\mathbb{C}}
\newcommand{\N}{\mathbb{N}}
\newcommand{\Z}{\mathbb{Z}}
\newcommand{\ket}[1]{\left|#1\right>}

\renewcommand{\[}{\begin{equation}}
\renewcommand{\]}{\end{equation}}

\makeatletter
\def\@tocline#1#2#3#4#5#6#7{\relax
  \ifnum #1>\c@tocdepth % then omit
  \else
    \par \addpenalty\@secpenalty\addvspace{#2}%
    \begingroup \hyphenpenalty\@M
    \@ifempty{#4}{%
      \@tempdima\csname r@tocindent\number#1\endcsname\relax
    }{%
      \@tempdima#4\relax
    }%
    \parindent\z@ \leftskip#3\relax \advance\leftskip\@tempdima\relax
    \rightskip\@pnumwidth plus4em \parfillskip-\@pnumwidth
    #5\leavevmode\hskip-\@tempdima
      \ifcase #1
       \or\or \hskip 1em \or \hskip 2em \else \hskip 3em \fi%
      #6 \hskip 0.5em \nobreak\relax
    \dotfill\hbox to\@pnumwidth{\@tocpagenum{#7}}\par
    \nobreak
    \endgroup
  \fi}
\makeatother

\begin{document}

\author[F.~D'Andrea]{Francesco D'Andrea} 
\address[F.~D'Andrea]{Dipartimento di Matematica e Applicazioni ``R.~Caccioppoli'',
	Universit\`a di Napoli Federico II, and INFN Sezione di Napoli, Via Cintia, 80126 Napoli, Italy.}
\email{francesco.dandrea@unina.it}
\author[P.M.~Hajac]{Piotr~M.~Hajac}
\address[P.M.~Hajac]{Instytut Matematyczny, Polska Akademia Nauk, ul.~\'Sniadeckich 8, Warszawa, 00-656 Poland}
\email{pmh@impan.pl}
\author[T.~Maszczyk]{Tomasz Maszczyk}
\address[T.~Maszczyk]{Instytut Matematyki, 
     Uniwersytet Warszawski,
ul.~Banacha 2,
02-097 Warszawa, Poland}
\email{t.maszczyk@uw.edu.pl}
\author[B.~Zieli\'nski]{Bartosz Zieli\'nski}
\address[B.~Zieli\'nski]{Department of  Computer Science, University of \L{}\'od\'z, Pomorska 149/153 90-236
\L{}\'od\'z, Poland} \email{bzielinski@uni.lodz.pl}
\title[The K-theory of a quantum projective plane]{\vspace*{-15mm}
Milnor meets Hopf and Toeplitz at\\[5pt] the K-theory of quantum projective planes}
\date{December 2025}

\vspace{-1cm}

\begin{abstract}
We explore applications of the celebrated construction of the Milnor connecting homomorphism from the odd to the even K-groups in the context of
Hopf--Galois theory. For a finitely generated projective module associated to any piecewise cleft principal comodule algebra, we provide an explicit formula 
computing the clutching $K_1$-class   in terms of the representation matrix defining the module. Thus, the module is determined by an explicit Milnor
idempotent. We apply this new tool  to the K-theory of quantum complex projective planes to determine their $K_0$-generators in terms of
modules associated to noncommutative Hopf fibrations. On the other hand, using explicit homotopy between unitaries, 
we express the $K_0$-class of the Milnor idempotents in terms of elementary projections in the Toeplitz C*-algebra.
This allows us to infer that all our generators are in the positive cone of the $K_0$-group, 
which is a purely quantum phenomenon absent in the classical case.
\end{abstract}

\maketitle

{\footnotesize\parskip=2pt\tableofcontents}
 
\section{Introduction}
\noindent
Free generators of the $K^0$-group of a complex projective space were implicitly provided by Atiyah and Todd in \cite{at60}.
Then, Adams in \cite{a-jf62}, explicitly rewrote the $K^0$-group as
\begin{equation}\label{eq:11}
K^{0}(\mathbb{C}P^n)=\mathbb{Z}[x]/(x^{n+1}) .
\end{equation}
Following \cite[Theorem 2.5]{karoubi2009k}, here we adopt the choice of $x$ to be the Euler class $[1]-[\mathrm{L}_1]$ of the dual $\mathrm{L}_1$ of the 
tautological line bundle $\mathrm{L}_{-1}$. 
The generators $1,x,\ldots,x^n$ of the $K^0$-group can be expressed in terms of line bundles  
$\mathrm{L}_k:=\mathrm{L}_1^{\otimes k}$, $k\in\mathbb{Z}$.
In particular, for $n=2$ these generators are:
\begin{equation}\label{eq:12}
1=[1],\qquad
x=[1]-[\mathrm{L}_1], \qquad
x^2 =[\mathrm{L}_{1}\oplus \mathrm{L}_{-1}] -2[1].
\end{equation}

Let us stress that in \cite{at60}, arguments are based on the  ring  structure of K-theory, the multiplicative Chern character, 
rational cohomology rings, and contracting   
projective hyperplanes to obtain even-dimensional spheres:
 \mbox{$\mathbb{C}P^n/ \mathbb{C}P^{n-1}\cong S^{2n}$}.  
None of these tools are available in noncommutative geometry. 

However,  a concept that is still available in noncommutative geometry is that of an associated vector bundle. 
In the classical case, we can view the above line bundles as associated with the principal $U(1)$-bundle $S^{2n+1}\to\mathbb{C}P^n$.
For $n=2$, we can view $\mathrm{L}_{1}\oplus \mathrm{L}_{-1}$ as the vector bundle associated via the fundamental representation of $SU(2)$ 
to the $SU(2)$-principal bundle $S^{5}\times^{U(1)}SU(2)\to \mathbb{C}P^2$.

On the other hand, another concept that is still available in noncommutative geometry is that of the Milnor connecting
homomorphism, transferred to the K-theory of C*-algebras by Higson, as explained in~\cite[Section~0.4]{hrz13}.
Using the standard CW-complex structure of $\C P^2$, we can see the third generator in \eqref{eq:12} as the image under the Milnor connecting 
homomorphism of the $K^1(S^3)$-class of the fundamental representation of~$SU(2)$, viewed as a continuous function from $S^3$
to the matrix algebra $M_2(\mathbb{C})$.

In the noncommutative setting of the Vaksman--Soibelman odd quantum spheres and quantum complex projective spaces, we take the fundamental 
representation $U$ of $SU_q(2)$, which  by the index pairing is proven to be a generator of
$K_1(C(SU_q(2)))\cong\mathbb{Z}$ in~\cite{dhhmw12}, to compute the third generator of $K_0(C(\mathbb{C}P_q^2))$ using our general
Milnor--Hopf--Galois result (Theorem~\ref{thm41}). Then we combine \eqref{eq-large-diag} with \cite[Theorem~5.1]{hnpsz} to transfer this result to the
setting of multipushout quantum complex projective spaces $\mathbb{C}P^n_H$ introduced in~\cite{hkz12}, and thus complete the calculation of 
generators of $K_0(C(\mathbb{C}P^2_H))$ started in~\cite{hr17}.
It is worth recalling here that $\mathbb{C}P^1_H$ is the mirror quantum sphere~\cite{hms06b}, and that the K-theory groups of
$\mathbb{C}P_H^n$ were determined in~\cite{r-j12} for $n=2$, and in the full generality in~\cite{hnpsz}. 

On the other hand, it is shown in \cite{masuda1990noncommutative} that a certain unitary $\widetilde{w}\in C(SU_q(2))$  generates $K_1(C(SU_q(2)))$.
The same unitary was obtained in \cite{dhhmw12} as the image of the Loring idempotent \cite{l-ta86}
under the Bott connecting homomorphism induced by the pullback diagram~\eqref{eq:pullsuq2}. We show by a homotopy argument
that $[\widetilde{w}]=[U]\in K_1(C(SU_q(2)))$ thus ensuring that the Milnor connecting homomorphism applied to $[\widetilde{w}]$ computes
the same generator in $K_0(C(\mathbb{C}P_q^2))$. Then we transfer this result to the setting of the multipushout quantum projective plane 
by combining \eqref{ProjectiveMapsOmegaSmall} with Theorem~\ref{thm:transfer}. We thus obtain the unitary 
${w}\in C(S^3_H)$  that generates $K_1(C(S^3_H))$ and computes the same generator of $K_0(C(\mathbb{C}P^2_H))$ via the Milnor
connecting homomorphism as was obtained using~\cite[Theorem~5.1]{hnpsz}. Remarkably, both of the thus computed generators are expressed
in terms of projections in the Toeplitz C*-algebra~$\mathcal{T}$, which can be viewed as the universal C*-algebra generated by an isometry~$s$ 
\cite{coburn1967c,coburn1969c}. 

Summarizing, we arrive at one of our main results:
\begin{thm} \label{mainresult}
Let $L_{-1}$ and $L_{1}$ denote the respective
section modules of noncommutative tautological and dual tautological line bundles over  the multipushout quantum complex projective plane 
$\mathbb{C}P_H^2$ as defined in \eqref{spectral},  and $\mathrm{p}:= (0, (1-ss^{*})\otimes (1-ss^{*}))$ be a Toeplitz-type projection in the C*-algebra  
$C(\mathbb{C}P_H^2)$.  Then the following double equality holds:  
\begin{equation}\label{maineq}
[L_1\oplus L_{-1}] -2[1] = \partial_{10}([w])= [\mathrm{p}].
\end{equation}
\end{thm}
\noindent
The interplay between group representations, topology and operator algebras described as in \eqref{maineq},
all contributing to unravelling the different facets of the K-theory of $\mathbb{C}P_H^2$, can be subsumed as follows: 
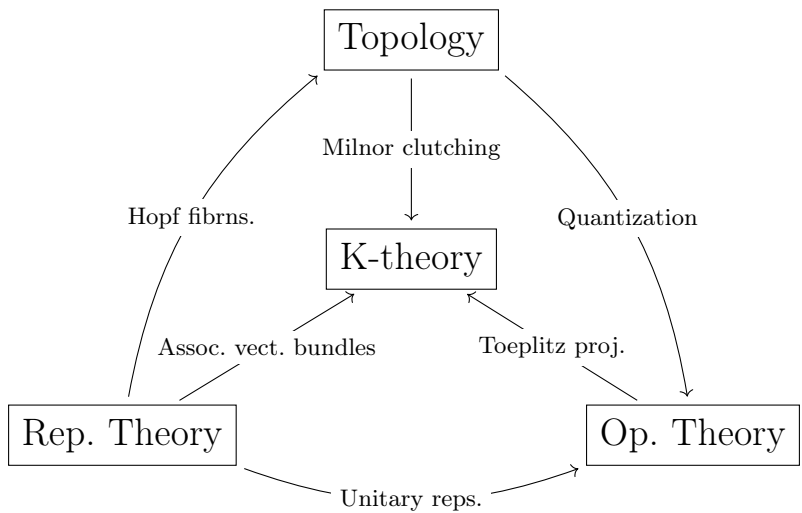
\begin{center}
\begin{tikzpicture}[font=\scriptsize]

\begin{scope}[font=\large,inner sep=5pt]
\node[draw] (a) at (0,0) {K-theory};
\node[draw,above] (b) at (90:2.5) {Topology};
\node[draw,below right] (c) at (-40:3) {Op.~Theory};
\node[draw,below left] (d) at (220:3) {Rep.~Theory};
\end{scope}

\begin{scope}[-To,shorten >=3pt,shorten <=3pt]
\path (b) edge node[fill=white] {Milnor clutching} (a);
\path (c) edge node[fill=white] {Toeplitz proj.} (a);
\path (d) edge node[fill=white] {Assoc.~vect.~bundles} (a);
\path (b.south east) edge[bend left=20] node[fill=white] {Quantization} (c);
\path (d.south east) edge[bend right=20] node[fill=white] {Unitary reps.} (c.south west);
\path (d) edge[bend left=20] node[fill=white] {Hopf fibrns.} (b.south west);
\end{scope}

\end{tikzpicture}
\end{center}

Furthermore, recall that in \cite{hr17},  the $K_0$-class $[L_1]-[1]$ was determined as one of free generators of 
\mbox{$K_0(C(\mathbb{C}P_H^2))\cong\mathbb{Z}^3$} 
in terms of projections in the  Toeplitz C*-algebra~$\mathcal{T}$.
Combining it with Theorem~\ref{mainresult},
we obtain the complete description of the free generators of $K_0(C(\mathbb{C}P_H^2))$ both in terms of associated projective modules
(the Atiyah--Todd picture) and  projections in  $\mathcal{T}$ (the Toeplitz picture).
\begin{cor}\label{AT-T}
The Atiyah--Todd-type and the  Toeplitz-type free generators of the group $K_0(C(\mathbb{C}P_H^2))\cong \mathbb{Z}^3 $  
are related as follows:
\begin{align*}
[1]&=\big[(1,1,1)\big],\\
[1]-[L_1] &=\big[(1-ss^*)\otimes 1, 0, 1\otimes (1-ss^*)\big] ,\\
[L_1\oplus L_{-1}]-2[1] &= \big[\big(0, 0, (1-ss^{*})\otimes (1-ss^{*})\big)\big].
\end{align*}
\end{cor}
\noindent
An immediate bonus of the right-hand side presentation of the generators is the fact that they are all in the positive cone of $K_0(C(\mathbb{C}P_H^2))$,
which is impossible in the classical case as the only rank-zero vector bundle is the zero bundle.
Furthermore, using the index pairing, we show that neither
$C(\C P^2_H)\mathrm{p}$ nor $C(\C P^2_H)(1-\mathrm{p})$ are modules associated to
the Peter--Weyl comodule algebra~$\mathcal{P}_{U(1)}(C(S^5_H))$.

Observe also that the $K_{0}$-class of the projection $1-ss^*$ expressing $K_0(C(\mathbb{C}P_H^2))$
is the nontrivial K-theory index of the Fredholm operator $s^*\in \mathcal{T}$.
To compare this phenomenon with the classical setting, we 
 identify $\mathcal{T}$ as the C*-algebra $C(D_q)$ of the quantum disc~\cite{kl93}.  
 Unlike the classical disc, the quantum disc is noncontractible~\cite{dhn20}, and affords nontrivial projections.
In contrast with this, the classical disc provides no non-trivial projections. However, to the classical complex projective plane, we can apply
the Atiyah--J{\"a}nich theorem (\cite{aj65}, \cite[Theorem~A.1]{ma67}) to conclude that
$K^{0}(\mathbb{C}P^{2})=[\mathbb{C}P^{2}, \mathcal{F}]$. Therefore, as $\mathbb{C}P^{2}$ can be built by Cartesian products followed by
 iterated one-side-cofibration pushouts with  discs as contractible building blocks, 
the right-hand side can be expressed by the $K$-theory indices of Fredholm operators, one associated with every disc. Finally, note that, while the
 left-hand side of the last two equations in Corollary~\ref{AT-T}
 admits a well-defined classical limit, the 
projections appearing on the right-hand side make  sense only in the quantum case. 

The paper is organized as follows. We start with the basics of quantum groups and Hopf algebras.
 Next, we recall the construction of Milnor idempotent and prove two  elementary general results in K-theory  needed later on.
 We end the preliminaries by unravelling the equivariant aspects of pullback diagrams.
 In the next section, we recollect definitions of  the plethora of quantum spaces appearing in our paper explainng their pushout structures.
 In the fourth section, we prove our main general result providing a formula for the modules associated to piecewise cleft principal 
 comodule algebras. Then we complete the section by computing the third generator of both $K_0(C(\mathbb{C}P_q^2))$ and $K_0(C(\mathbb{C}P_H^2))$
 in terms of associated modules. We end the paper by computing these generators in terms of elementary projections in the Toeplitz 
 C*-algebra~$\mathcal{T}$.

\section{Preliminaries}

\subsection{Compact quantum groups and principal comodule algebras}
In what follows, we will frequently use actions and coactions.
Let $(H,\Delta)$ be a compact quantum group, and $A$ be a C*-algebra. An action of a compact quantum group $(H,\Delta)$ on $A$ is given
by a right $H$-coaction $\delta_R:A\to A\otimes_{\mathrm{min}} H$.  We refer to C*-algebras endowed with an action of $(H,\Delta)$ 
as $H$-C*-algebras. If 
$H=C(G)$, where $G$ is a compact Hausdorff group, we will also use the name $G$-C*-algebra.

In what follows, all our C*-tensor products will be with nuclear C*-algebras, so due to the lack of ambiguity we 
simplify the notation using $\otimes$ to stand for the C*-tensor product. However, whenever $\otimes$ is used not in between two C*-algebras,
it means  the algebraic tensor product.

Recall that $H=C(G)$ has a comultiplication, counit and antipode obtained by dualizing the product, unit and inverse in $G$. Explicitly,
\begin{align}
\Delta &:H\longrightarrow H\otimes H\cong C(G\times G) &
\varepsilon &:H\longrightarrow \C &
S &:H\longrightarrow H \nonumber\\
\Delta &(f)(g_1,g_2) :=f(g_1g_2)
&
\varepsilon &(f) :=f(e)
&
S &(f)(g) :=f(g^{-1})
\end{align}
for all $f\in H$, $g,g_1,g_2\in G$ and with $e$ denoting the neutral element.

Given a strongly continuous action $\alpha$ of a compact group $G$ on a C*-algebra $A$, the dual right coaction is defined by 
\begin{equation}\label{eq:22}
\setlength{\arraycolsep}{2pt}
\delta_R\colon A\longrightarrow\!\!
\begin{array}[t]{rl}A\otimes C(G) &\xrightarrow{\;\; \cong \;\;}C(G,A) \\
a\otimes f &\longmapsto \big\{ g\mapsto a\,f(g) \big\}, \end{array}
 \qquad
\delta_R(a)(g):=\alpha_g(a) \;\forall\; a\in A,\,g\in G .
\end{equation}

Next, let $k$ be a field and $\mathcal{H}$ be a Hopf algebra over $k$.
The cotensor product of a right comodule $M$ with a coaction 
$\delta_R:M\to M\otimes \mathcal{H}$ with a left comodule $N$ with a coaction 
$\delta_L:N\to \mathcal{H}\otimes N$ is defined as the difference kernel
\begin{equation}
M\mathbin{\Box}^{\mathcal{H}} N:=\ker\Big(\delta_R\otimes\id-\id\otimes\delta_L\colon M\otimes N
\longrightarrow M\otimes \mathcal{H}\otimes N\Big).
\end{equation}

We also need to recall the notion of a \emph{principal comodule algebra} \cite{hajac2011piecewise} and the concept of a 
\emph{strong connection}~\cite{hajac1996strong}. Let $\mathcal{P}$ be a right $\mathcal{H}$-comodule algebra and 
$\mathcal{B}:=\mathcal{P}^{\co\mathcal{H}}$ be the coaction-invariant subalgebra. Now, the canonical map 
\[
\mathrm{can}:\mathcal{P}\otimes_{\mathcal{B}}\mathcal{P}\longrightarrow \mathcal{P}\otimes\mathcal{H} , \qquad
p\otimes p' \longmapsto p\hspace{1pt}p'_{(0)}\otimes p'_{(1)} ,
\]
is a left $\mathcal{P}$-module and right $\mathcal{H}$-comodule map with respect to the multiplication of $\mathcal{P}$ on the left factors and 
the coaction of $\mathcal{H}$ on the right factors. (Here we use the Heyneman--Sweedler notation $p'_{(0)}\otimes p'_{(1)}$ 
with the summation sign suppressed for the right coaction
applied to~$p'$.)
We say that $\mathcal{P}$ is principal if the canonical map is bijective and there exists a strong connection, i.e.~a unital bi-colinear map 
$\ell:\mathcal{H}\to\mathcal{P}\otimes\mathcal{P}$ satisfying
\[
\mathrm{can}\circ\ell=1_{\mathcal{P}}\otimes \id_{\mathcal{H}} .
\]

We are now ready to show that, if $V$ is a left $\mathcal{H}$-comodule and $\mathcal{P}$ admits a character, 
then the ``fiber'' of the left $\mathcal{B}$-module $\mathcal{P}\mathbin{\Box}^{\mathcal{H}}V$
is isomorphic to~$V$. (Much as above, in what follows, we will use the Heyneman--Sweedler notation $v_{(-1)}\otimes v_{(0)}$ 
with the summation sign suppressed for the left
coaction applied to~$v\in V$.)
\begin{prop}\label{prop:21}
Let $\mathcal{P}$ a principal $\mathcal{H}$-comodule algebra admitting a character, $V$ be a left $\mathcal{H}$-comodule, and $\mathcal{B}$ the 
coaction-invariant subalgebra. 
Then the vector space $k\otimes_{\mathcal{B}}(\mathcal{P}\mathbin{\Box}^{\mathcal{H}}V)$ is isomorphic to~$V$.
\end{prop}
\begin{proof}
Using the assumptions, we compute:
\begin{align}
k\otimes_{\mathcal{B}}(\mathcal{P}\mathbin{\Box}^{\mathcal{H}}V) &\cong
k\otimes_{\mathcal{P}}\big(\mathcal{P}\otimes_{\mathcal{B}}(\mathcal{P}\mathbin{\Box}^{\mathcal{H}}V)\big) \nonumber\\
&\cong
k\otimes_{\mathcal{P}}\big((\mathcal{P}\otimes_{\mathcal{B}}\mathcal{P})\mathbin{\Box}^{\mathcal{H}}V\big) \nonumber\\
&\cong
k\otimes_{\mathcal{P}}\big((\mathcal{P}\otimes\mathcal{H})\mathbin{\Box}^{\mathcal{H}}V\big) \nonumber\\
&\cong
k\otimes_{\mathcal{P}}(\mathcal{P}\otimes V) \nonumber\\
&\cong V .
\end{align}
In the second step, we used the flatness of $\mathcal{P}$ over $\mathcal{B}$ guaranteed by the principality of~$\mathcal{P}$, and the third isomorphism 
is induced by the canonical map.
\end{proof}

\subsection{Pullback diagrams of algebras}\label{Milnor}

Although the term  ``pullback'' is used to denote different things in distinct contexts, herein it stands only for the limit of diagrams of the form 
$A_1\xrightarrow{\;\;\;\pi_1\,\;\;}A_{12}
\xleftarrow{\;\;\,\pi_2\;\;\;}A_2$.
The limit of $A_1\xrightarrow{\;\;\;\pi_1\,\;\;}A_{12}
\xleftarrow{\;\;\,\pi_2\;\;\;}A_2$  is a universal pair of morphisms  
$A_1\xleftarrow{\;\;\;p_1\,\;\;}A\xrightarrow{\;\;\,p_2\;\;\;}A_2$
 such that $\pi_1\circ p_1=\pi_2\circ p_2$. Here, the universality means that, for any other pair of morphisms  
 $A_1\xleftarrow{\;\;\;k_1\,\;\;}Z\xrightarrow{\;\;\,k_2\;\;\;}A_2$ such that $\pi_1\circ k_1=\pi_2\circ k_2\,$, there exists a unique morphism
$\langle k_1,k_2\rangle: Z\rightarrow A$ such that
$p_i\circ\langle k_1,k_2\rangle = k_i$, $i=1,2$. All categories we use in this paper (e.g., unital C*-algebras, $H$-C*-algebras, modules, comodules) 
admit pullbacks. 

It is a general fact in category theory that a morphism of diagrams induces a morphism of limits. 
We will use this fact multiple times in the following special case.
Assume that we have two pullbacks in a category $\mathcal{C}$: 
\begin{itemize}
\item a pullback 
$A_1\xleftarrow{\;\;\;p_1\,\;\;}A\xrightarrow{\;\;\,p_2\;\;\;}A_2$
of
$A_1\xrightarrow{\;\;\pi_1\;\;}A_{12}\xleftarrow{\;\;\pi_2\;\;}A_2$,  and 
\item a pullback
$B_1\xleftarrow{\;\;\;q_1\,\;\;}B\xrightarrow{\;\;\,q_2\;\;\;}B_2$
of
$B_1\xrightarrow{\;\;\rho_1\;\;}B_{12}\xleftarrow{\;\;\rho_2\;\;}B_2$.
\end{itemize}
Furthermore, assume that we have
morphisms $\phi_1: A_1 \to B_1$, $\phi_2:A_1 \to B_2$ and
\mbox{$\phi_{12}:A_{12}\to B_{12}$} such that the below solid-line diagram  commutes:
\begin{equation}\label{twopullcommthm}
\begin{tikzpicture}[scale=1.8]

\node (A) at (1,2) {$A$};
\node (B) at (0,1) {$A_1$};
\node (C) at (2,1) {$A_2$};
\node (D) at (1,0) {$A_{12}$};
\node (E) at (5.5,2) {$B$};
\node (F) at (4.5,1) {$B_1$};
\node (G) at (6.5,1) {$B_2$.};
\node (H) at (5.5,0) {$B_{12}$};

\path[->,font=\scriptsize,inner sep=5pt]
	(A) edge node[left] {$p_1$} (B)
	(A) edge node[right] {$p_2$} (C)
	(B) edge node[left] {$\pi_1$} (D)
	(C) edge node[pos=0.8,right] {$\pi_2$} (D)
	(E) edge node[pos=0.2,left] {$q_1$} (F)
	(E) edge node[right] {$q_2$} (G)
	(F) edge node[left] {$\rho_1$} (H)
	(G) edge node[right] {$\rho_2$} (H)
	(A) edge[bend left=10,dashed] node[above] {$\phi$} (E)
	(B) edge[bend right=15] node[below,pos=0.6] {$\phi_1$} (F)
	(C) edge[bend left=15] node[above,pos=0.4] {$\phi_2$} (G)
	(D) edge[bend right=10] node[below] {$\phi_{12}$} (H);
			
\end{tikzpicture}
\end{equation}
Then there exists a unique map 
$\phi: A\rightarrow B$ which makes the whole diagram \eqref{twopullcommthm} commute.

When $A_1\xrightarrow{\;\;\pi_1\;\;}A_{12}\xleftarrow{\;\;\pi_2\;\;}A_2$ is a diagram in the category of sets, abelian groups,
vector spaces, modules, comodules, algebras, and the like, its limit is given by the equalizer
\[
A=A_1\times_{A_{12}}A_2 :=\big\{(a_1,a_2)\;\big|\;\pi_1(a_1)=\pi_2(a_2) \big\} ,
\]
with maps $A_1\xleftarrow{\;\;\mathrm{pr}_1\;\;}A\xrightarrow{\;\;\mathrm{pr}_2\;\;}A_2$ given by the canonical projections onto the two factors. 
In \eqref{twopullcommthm}, if both $A=A_1\times_{A_{12}}A_2$ and $B=B_1\times_{B_{12}}B_2$, then the map $\phi$ is given by
\begin{equation}\label{eq:mapphi}
\phi(a_1,a_2)=\big(\phi_1(a_1),\phi_2(a_2)\big) .
\end{equation}

We will now recall the clutching construction of a noncommutative vector bundle~\cite{m-j71,bdh15,dhhmw12}.
It is given in terms of the pullback of
C*-algebras and *-homomorphisms $\rho_1$ and $\rho_2$
such that at least one of them is surjective.
For starters, given a pullback diagram of C*-algebras
\begin{equation}\label{eq-pullback2}
\begin{tikzpicture}[scale=0.8]

\node (A) at (2,3.6) {$B$};
\node (B) at (0,1.8) {$B_1$};
\node (C) at (4,1.8) {$B_2$};
\node (D) at (2,0) {$B_{12}$};

\path[->,font=\scriptsize,inner sep=2pt] (A) edge node[above left] {$q_1$} (B)
	(A) edge node[above right] {$q_2$} (C)
	(B) edge node[below left] {$\rho_1$} (D)
	(C) edge[->>] node[below right] {$\rho_2$} (D);
			
\end{tikzpicture}
\end{equation}
with $\rho_2$ surjective, we recall that there is an associated six-term exact sequence in K-theory~\cite{hilgert1986mayer}:
\begin{equation}\label{eq:verify}
\begin{tikzpicture}[xscale=4,yscale=2.3]

\node (a) at (0,1) {$K_0(B)$};
\node (b) at (1,1) {$K_0(B_1)\oplus K_0(B_2)$};
\node (c) at (2,1) {$K_0(B_{12})$};
\node (d) at (0,0) {$K_1(B_{12})$};
\node (e) at (1,0) {$K_1(B_1)\oplus K_1(B_2)$};
\node (f) at (2,0) {$K_1(B).$};

\path[->,font=\scriptsize]
	(a) edge node[above] {$(q_{1*},q_{2*})$} (b)
	(b) edge node[above] {$\rho_{2*}-\rho_{1*}$} (c)
	(c) edge node[right] {$\partial_{01}$} (f)
	(f) edge node[above] {$(q_{1*},q_{2*})$} (e)
	(e) edge node[above] {$\rho_{2*}-\rho_{1*}$} (d)
	(d) edge node[left] {$\partial_{10}$} (a);

\end{tikzpicture}
\end{equation}
An explicit formula for the odd-to-even connecting homomorphism $\partial_{10}$ is due to Milnor and is defined as follows.
Given $[a]\in K_1(B_{12})$ with a representative $a\in GL_n(B_{12})$,
choose $c, d\in M_n(B_2)$ such that
$M_n(\rho_2)(c)=a$ and $M_n(\rho_2)(d)=a^{-1}$. (Here $M_n(\rho_2)$ is $\rho_2$ applied entrywise to matrices.)
Then $\partial_{10}([a])=[p_a]-[I_n]$, 
where (see \cite[Theorem~2.2]{dhhmw12}):
\begin{equation}
\label{MilnProjEq}
p_a:=\left(
\begin{array}{cc}
(1, c(2-dc)d) & (0, c(2-dc)(1-dc))\\
(0, (1-dc)d) & (0, (1-dc)^2)
\end{array}
\right) \in M_{2n}(B) .
\end{equation}
This projection has the following interpretation in terms of the pullback of modules.
Let $V$ be a finite-dimensional complex vector space and
$\chi: B_{12}\otimes V\rightarrow B_{12}\otimes V$ be an isomorphism of left $B_{12}$-modules. We construct a finitely generated
projective left $B$-module \mbox{$M(B_1\otimes V, B_2\otimes V, \chi)$} as follows (see \cite{dhhmw12}).
It is defined  as the pullback of the free left
$B_1$-module $B_1\otimes V$ and the free left $B_2$-module $B_2\otimes V$:
\begin{equation}\label{eq-vector-bundle}
\begin{tikzpicture}[xscale=3.5,yscale=1.7]

\node (A) at (0,2) {$M(B_1\otimes V, B_2\otimes V, \chi)$};
\node (B) at (-1,1) {$B_1\otimes V$};
\node (C) at (1,1) {$B_2\otimes V$};
\node (D) at (-1,0) {$B_{12}\otimes V$};
\node (E) at (1,0) {$B_{12}\otimes V.$};

\path[->,font=\scriptsize] (A) edge (B.north east)
	(A) edge  (C.north west)
	(B) edge node[left] {$\rho_1\otimes\id$} (D)
	(C) edge node[right] {$\rho_2\otimes\id$} (E)
	(D) edge node[above] {$\chi$} (E);
			
\end{tikzpicture}
\end{equation}

\begin{thm}[\protect{\cite[Theorem~2.1]{dhhmw12}}]
\label{thm-Milnor}
Consider the pullback diagram \eqref{eq-pullback2}, with $\rho_2$ surjective.
Let $n:=\dim V$, let $E:=M(B_1\otimes V, B_2\otimes V, \chi)$ be the left $B$-module  in \eqref{eq-vector-bundle} and let $a \in GL(n,B_{12}) $ be the 
representative matrix of the isomorphism
$$
\chi:B_{12}\otimes V\longrightarrow B_{12}\otimes V 
$$
 in \eqref{eq-vector-bundle}. Then $E\cong B^{2n}p_a$ as left $B$-modules, where $p_a$ is the idempotent \eqref{MilnProjEq}.
\end{thm}

Next, let us recall an explicit formula for the even-to-odd connecting homomorphism $\partial_{01}$, which is due to Bott. Let $[p]\in K_0(B_{12})$ 
be represented by a projection $p\in M_k(B_{12})$. Choose a selfadjoint lifting $Q$ of $p$ to $B_2$ (i.e.~$M_k(\rho_2)(Q)=p$). 
Then (see \cite[Theorem~2.3]{dhhmw12}):
\begin{equation}\label{eq:Bottpartial}
\partial_{01}[p]=[(I_k,e^{2\pi iQ})] \in K_1(B) .
\end{equation}
With this, we can now prove the following simple, possibly well known and certainly extremely useful, result.
\begin{thm}\label{thm:transfer}
Assume that we have a commutative diagram of C*-algebras as in \eqref{twopullcommthm}, with $\pi_2$ and $\rho_2$ surjective. Then we have the 
following induced commutative diagram in K-theory:
\begin{center}
\begin{tikzpicture}[xscale=4.5,yscale=2.3,inner sep=2pt]

\node (a) at (0.4,1) {$K_0(A)$};
\node (b) at (1.4,1) {$K_0(A_1)\oplus K_0(A_2)$};
\node (c) at (2.4,1) {$K_0(A_{12})$};
\node (d) at (0,0.2) {$K_1(A_{12})$};
\node (e) at (1,0.2) {$K_1(A_1)\oplus K_1(A_2)$};
\node (f) at (2,0.2) {$K_1(A)$};

\node (g) at (0.4,-1) {$K_0(B)$};
\node (h) at (1.4,-1) {$K_0(B_1)\oplus K_0(B_2)$};
\node (i) at (2.4,-1) {$K_0(B_{12})$};
\node (l) at (0,-1.8) {$K_1(B_{12})$};
\node (m) at (1,-1.8) {$K_1(B_1)\oplus K_1(B_2)$};
\node (n) at (2,-1.8) {$K_1(B)$.};

\path[-To,font=\footnotesize]
	(a) edge node[fill=white,pos=0.39,inner sep=4pt] {} node[right,pos=0.65] {$\phi_*$} (g)
	(b) edge node[fill=white,pos=0.39,inner sep=4pt] {} node[right,pos=0.65] {$(\phi_{1*},\phi_{2*})$}  (h)
	(g) edge[shorten >=-2pt] node[fill=white,pos=0.93,inner sep=3pt] {} node[above,pos=0.4] {$(q_{1*},q_{2*})$} (h)
	(h) edge node[fill=white,pos=0.51,inner sep=4pt] {} node[above,pos=0.6] {$\rho_{2*}-\rho_{1*}$} (i)
	(c) edge node[right] {$\phi_{12*}$} (i)
	(d) edge node[left] {$\phi_{12*}$} (l)
	(e) edge node[right,pos=0.35] {$(\phi_{1*},\phi_{2*})$} (m)
	(f) edge node[right,pos=0.35] {$\phi_*$} (n);

\path[-To,font=\footnotesize]
	(a) edge node[above] {$(p_{1*},p_{2*})$} (b)
	(b) edge node[above] {$\pi_{2*}-\pi_{1*}$} (c)
	(c) edge node[below right] {$\partial_{01}$} (f)
	(f) edge node[above,pos=0.3] {$(p_{1*},p_{2*})$} (e)
	(e) edge node[above,pos=0.43] {$\pi_{2*}-\pi_{1*}$} (d)
	(d) edge node[above left] {$\partial_{10}$} (a);

\path[-To,font=\footnotesize]
	(i) edge node[below right] {$\partial_{01}$} (n)
	(n) edge node[above] {$(q_{1*},q_{2*})$} (m)
	(m) edge node[above] {$\;\rho_{2*}-\rho_{1*}$} (l)
	(l) edge node[above left] {$\partial_{10}$} (g);

\end{tikzpicture}
\end{center}
Here the top and bottom faces in this diagram are the six-term exact sequences of the two pullback squares in \eqref{twopullcommthm}.
\end{thm}
\begin{proof}
Without the loss of generality, we can assume that $A=A_1\times_{A_{12}}A_2$ and $B=B_1\times_{B_{12}}B_2$. We must show that the connecting 
homomorphisms commute with the vertical maps.

For the odd-to-even one, let $[a]\in K_1(A_{12})$. Choose $c,d$ such that $M_n(\rho_2)(c)=a$ and $M_n(\rho_2)(d)=a^{-1}$. Let $p_a$ be the 
idempotent \eqref{MilnProjEq}. Next, note that $\phi_{12*}([a])=[a']$, where $a':=M_n(\phi_{12})(a)$. 
Now, let $c':=M_n(\phi_2)(c)$ and $d'=M_n(\phi_2)(d)$. By the 
commutativity of the diagram \eqref{twopullcommthm},
one has $M_n(\rho_2)(c')=a'$ and $M_n(\rho_2)(d')=a'^{-1}$. The corresponding 
idempotent $p_{a'}$ in \eqref{MilnProjEq} satisfies $M_{2n}(\phi)(p_a)=p_{a'}$ and,
since the K-theory classes do not depend on any of the choices made, we see that
\begin{equation}
\phi_*(\partial_{10}([a]))=
[M_{2n}(\phi)(p_a)]-[I_n]=[p_{a'}]-[I_n]=\partial_{10}([a'])=\partial_{10}(\phi_{12*}[a]) ,
\end{equation}
that is
$\phi_*\circ\partial_{10}=\partial_{10}\circ\phi_{12*}$.

For the even-to-odd one, we proceed much in the same way. Given $[p]\in K_0(A_{12})$,  we lift its representative $p\in M_k(A_{12})$
to a selfadjoint 
$Q\in M_k(A_2)$, and call $p':=M_k(\phi_{12})(p)$ and $Q'=M_k(\phi)(Q)$. Since $Q'$ is a selfadjoint lift of $p'$, using \eqref{eq:Bottpartial},
 we compute
\begin{equation}
\partial_{01}\circ\phi_{12*}([p])=\partial_{01}([p'])=[(I_k,e^{2\pi iQ'})]=
\phi_*( [(I_k,e^{2\pi iQ})] )=\phi_*\circ\partial_{01}([p]) .
\end{equation}
Hence, $\partial_{01}\circ\phi_{12*}=\phi_*\circ\partial_{01}$.
\end{proof}

We will now derive a fundamental isomorphism property for the $K$-theory of  pullbacks.
\begin{thm}\label{thm:3outof4}
\label{prop-funct-pull-back-K} Assume that we have a commutative diagram in the category of C*-algebras as in \eqref{twopullcommthm}, 
with $\pi_2$ and $\rho_2$ surjective.
Assume also that the *-homorphisms $\phi_1$, $\phi_2$ and $\phi_{12}$ induce isomorphisms of $K$-groups:
$$
\phi_{i\ast}:K_*(A_i)\xrightarrow{\;\;\cong\;\;}K_*(B_i),\quad
\phi_{12\ast}:K_*(A_{12})\xrightarrow{\;\;\cong\;\;}K_*(B_{12}).
$$
Then the morphism
$\phi$ also  induces an isomorphim in $K$-theory:
$$
\phi_* : K_*(A)\xrightarrow{\;\;\cong\;\;}K_*(B).
$$
\end{thm}

\begin{proof}
Without loosing generality, we can assume that $A=A_1\times_{A_{12}}A_2$, $B=B_1\times_{B_{12}}B_2$, and $\phi$ is given by~\eqref{eq:mapphi}.
We are under the hypothesis of Theorem~\ref{thm:transfer}, with two six-term-exact sequences connected by six morphisms
given by $\phi_{12*}$, $(\phi_{1*},\phi_{2*})$ and $\phi_*$ in both the even and the odd degree. The first two induced maps are isomorphisms.

To prove that also $\phi_*$ is  an isomorphism, in both the even and the odd degree, we will use the commutative diagram in 
Theorem~\ref{thm:transfer} and the 
5-lemma. First, we consider the subdiagram
\begin{center}
\begin{tikzpicture}[xscale=3,yscale=2]

\node (b1) at (1,1) {$K_0(A_1)\oplus K_0(A_2)$};
\node (c1) at (2,1) {$K_0(A_{12})$};
\node (d1) at (2.8,1) {$K_1(A)$};
\node (e1) at (3.8,1) {$K_1(A_1)\oplus K_1(A_2)$};
\node (f1) at (4.8,1) {$K_1(A_{12})$};

\node (b2) at (1,0) {$K_0(B_1)\oplus K_0(B_2)$};
\node (c2) at (2,0) {$K_0(B_{12})$};
\node (d2) at (2.8,0) {$K_1(B)$};
\node (e2) at (3.8,0) {$K_1(B_1)\oplus K_1(B_2)$};
\node (f2) at (4.8,0) {$K_1(B_{12})$.};

\path[->,font=\scriptsize]
	(b1) edge (c1)
	(c1) edge node[above] {$\partial_{01}$} (d1)
	(d1) edge (e1)
	(e1) edge (f1)
	(b2) edge (c2)
	(c2) edge node[below] {$\partial_{01}$} (d2)
	(d2) edge (e2)
	(e2) edge (f2)
	(b1) edge node[right] {$(\phi_{1*},\phi_{2*})$} (b2)
	(c1) edge node[right] {$\phi_{12*}$} (c2)
	(d1) edge node[left] {$\phi_*$} (d2)
	(e1) edge node[left] {$(\phi_{1*},\phi_{2*})$} (e2)
	(f1) edge node[left] {$\phi_{12*}$} (f2);

\end{tikzpicture}
\end{center}
Since the rows are exact, and the first two and last two vertical arrows are isomorphisms by hypothesis, it follows that $\phi_*:K_1(A)\to K_1(B)$ is an 
isomorphism from the 5-lemma.

Next, we look at the following subdiagram of the diagram in Theorem~\ref{thm:transfer}:
\begin{center}
\begin{tikzpicture}[xscale=3,yscale=2]

\node (a1) at (0,1) {$K_0(A)$};
\node (b1) at (1,1) {$K_0(A_1)\oplus K_0(A_2)$};
\node (c1) at (2,1) {$K_0(A_{12})$};
\node (e1) at (-2,1) {$K_1(A_1)\oplus K_1(A_2)$};
\node (f1) at (-1,1) {$K_1(A_{12})$};

\node (a2) at (0,0) {$K_0(B)$};
\node (b2) at (1,0) {$K_0(B_1)\oplus K_0(B_2)$};
\node (c2) at (2,0) {$K_0(B_{12})$};
\node (e2) at (-2,0) {$K_1(B_1)\oplus K_1(B_2)$};
\node (f2) at (-1,0) {$K_1(B_{12})$.};

\path[->,font=\scriptsize]
	(a1) edge (b1)
	(b1) edge (c1)
	(e1) edge (f1)
	(a2) edge (b2)
	(b2) edge (c2)
	(e2) edge (f2)
	(a1) edge node[right] {$\phi_*$} (a2)
	(b1) edge node[right] {$(\phi_{1*},\phi_{2*})$} (b2)
	(c1) edge node[right] {$\phi_{12*}$} (c2)
	(e1) edge node[left] {$(\phi_{1*},\phi_{2*})$} (e2)
	(f1) edge node[left] {$\phi_{12*}$} (f2)
	(f1) edge node[above] {$\partial_{10}$} (a1)
	(f2) edge node[below] {$\partial_{10}$} (a2);

\end{tikzpicture}
\end{center}
Again, since the rows are exact, and the first two and last two vertical arrows are isomorphisms by hypothesis,  it follows that 
$\phi_*:K_0(A)\to K_0(B)$ is an isomorphism from the 5-lemma.
\end{proof}

\subsection{Gauge trick for pullback diagrams}
Let $X$ be a compact Hausdorff space equipped with a continuous right action of a compact Hausdorff group $G$.
We regard $X\times G$ as a right $G$-space in two different ways, which we distinguish notationally as follows:
\begin{itemize}
\item We write $X\times G^\bullet$ for the  product $X\times G$ with the $G$-action \mbox{$(x,g) h := (x, gh)$}.
\item We write $X^\bullet\times G^\bullet$ for the same space with  the $G$-action $(x, g) h := (x h,gh)$.
\end{itemize}
Now, there is a $G$-equivariant homeomorphism
\[
\hat\kappa : X\times G^\bullet  \ni
(x, g) \longmapsto (x g, g) \in X^\bullet\times G^\bullet ,
\]
with its inverse given by $\hat\kappa^{-1}(x,g) := (x g^{-1},g)$.

Next, let $A$ be a unital $G$-C*-algebra and $H := C(G)$.
Much as above, using \eqref{eq:22} we equip the C*-algebra $A\otimes H$ with two right $H$-coactions:
\begin{equation}
\begin{split}
A^\bullet\otimes H^\bullet &\ni a\otimes h\longmapsto
a\sw{0}\otimes h\sw{1}\otimes a\sw{1}h\sw{2} \in (A^\bullet\otimes H^\bullet)\otimes H , \\
A\otimes H^\bullet &\ni a\otimes h\longmapsto a\otimes
h\sw{1}\otimes h\sw{2} \in (A\otimes H^\bullet)\otimes H .
\end{split}
\end{equation}
Then the following map is an $H$-equivariant isomorphism of C*-algebras:
\begin{equation}\label{kappa}
\kappa:A^\bullet\otimes H^\bullet \ni
a\otimes h\longmapsto a\sw{0}\otimes a\sw{1}h
\in A\otimes H^\bullet ,
\end{equation}
\noindent
with its inverse given by $\kappa^{-1}(a\otimes h):=a\sw{0}\otimes S(a\sw{1})h$.

We now explain how we will use these maps to ``gauge'' a pullback diagram. Consider a pullback diagram of $H$-C*-algebras
of the form
\begin{equation}
\begin{tikzpicture}[scale=1.6]

\node (A) at (1,2) {$A$};
\node (B) at (0,1) {$A_1$};
\node (C) at (2,1) {$A_2^\bullet\otimes H^\bullet$};
\node (D) at (1,0) {$A_{12}^\bullet\otimes H^\bullet$};

\path[->,font=\scriptsize,inner sep=1pt]
	(A) edge node[above left] {$p_1$} (B)
	(A) edge node[above right] {$p_2$} (C)
	(B) edge node[below left] {$\pi_1$} (D)
	(C) edge node[below right] {$\pi_2\otimes\id$} (D);
			
\end{tikzpicture} .
\end{equation}
We can transform it into an analogous diagram where the coactions on the bottom and right nodes are on the rightmost factors:
\begin{equation}\label{eq:gaugetrick}
\begin{tikzpicture}[scale=1.8]

\node (A) at (1,2) {$A$};
\node (B) at (0,1) {$A_1$};
\node (C) at (2,1) {$A_2^\bullet\otimes H^\bullet$};
\node (D) at (1,0) {$A_{12}^\bullet\otimes H^\bullet$};
\node (E) at (5.5,2) {$A$};
\node (F) at (4.5,1) {$A_1$};
\node (G) at (6.5,1) {$A_2\otimes H^\bullet$};
\node (H) at (5.5,0) {$A_{12}\otimes H^\bullet$};

\path[->,font=\scriptsize,inner sep=5pt]
	(A) edge node[left] {$p_1$} (B)
	(A) edge node[right] {$p_2$} (C)
	(B) edge node[left] {$\pi_1$} (D)
	(C) edge node[pos=0.8,right] {$\pi_2\otimes\id$} (D)
	(E) edge node[pos=0.2,left] {$p_1$} (F)
	(E) edge node[right] {$\kappa\circ p_2$} (G)
	(F) edge node[left] {$\kappa\circ\pi_1$} (H)
	(G) edge node[right] {$\pi_2\otimes\id$} (H)
	(A) edge[dashed,bend left=10] (E)
	(B) edge[bend right=15] node[below,pos=0.6] {$\id$} (F)
	(C) edge[bend left=15] node[above,pos=0.4] {$\kappa$} (G)
	(D) edge[bend right=10] node[below] {$\kappa$} (H);
			
\end{tikzpicture} .
\end{equation}
The dashed map is the isomorphism induced on pullbacks. Note that the commutativity of $\pi_2$ with $\kappa$ is 
a consequence of the $H$-equivariance of $\pi_2$. Observe also that the multiplication in $H$ is a *-homomorphism because $H$ is commutative.

\section{Quantum spaces as pushouts}

We denote by $u$ the unitary generator of $C(S^1)$ given by the inclusion of $S^1$ into $\C$, and by $s$ the generating isometry
of the Toeplitz C*-algebra~$\mathcal{T}$.
There is a natural $U(1)$-action $\alpha$ 
on $\mathcal{T}$ given by rephasing the isometry~$s$:
$\alpha_\lambda(s):=\lambda s\;\forall\;\lambda\in U(1)$.
There is a well-known short exact sequence of $U(1)$-C*-algebras
\begin{equation*}
0\longrightarrow\mathcal{K}\longrightarrow\mathcal{T}\xrightarrow{\ \sigma\ } C(S^1)\longrightarrow 0 ,
\end{equation*}
where $\sigma$ is the symbol map mapping $s$ to $u$,
$\mathcal{K}$ is the ideal of compact operators. Here we identify $S^1$ with $U(1)$, 
and the $U(1)$-coaction on $S^1$ is given by the group multiplication.

The goal of this section is to identify the needed C*-algebras of quantum spaces as pullback C*-algebras. In the next sections, we will abuse the notation by using the same symbol to denote a given C*-algebra and its pullback presentation.

\subsection{The Woronowicz compact quantum group $SU_q(2)$}
For $0<q<1$, we define $C(SU_q(2))$ as the universal C*-algebra given by generators $\alpha$ and $\gamma$ subject to the following relations~\cite{w-sl87}:
\begin{equation}
\alpha\gamma = q\gamma\alpha, \quad \alpha\gamma^*=q\gamma^*\alpha,
\quad \gamma\gamma^*=\gamma^*\gamma,\quad 
\alpha^*\alpha+\gamma^*\gamma=1,\quad 
\alpha\alpha^*+q^2\gamma\gamma^*=1.
\end{equation}
The fundamental corepresentation matrix of $SU_q(2)$ can be
written as:
\begin{equation}\label{funrep}
U:=\left(\begin{array}{cc}
\alpha&-q\gamma^{*}\\
\gamma &\alpha^{*}
\end{array}\right).
\end{equation}
The assignment $\alpha\mapsto\pi(\alpha):=u^*$, $\gamma\mapsto\pi(\gamma)=0$, defines a *-homomorphism
\begin{equation}
\pi:C(SU_q(2))\longrightarrow C(U(1)),
\end{equation}
which restricts to a homomorphism of Hopf algebras.
Now we can define both left and right coactions of $C(U(1))$ on $C(SU_q(2))$ as follows:
 \begin{equation}\label{eq:resu1}
 \begin{split}
 \delta_L&:=(\pi\otimes\id)\circ\Delta\colon C(SU_q(2))\longrightarrow
 C(U(1))\otimes C(SU_q(2)),\\
 \delta_R&:=(\id\otimes\pi)\circ\Delta\colon C(SU_q(2))\longrightarrow
 C(SU_q(2))\otimes C(U(1)),
 \end{split}
 \end{equation}
where $\Delta$ is the comultiplication.

Our opposite to the usual definition of $\pi$ is dictated by the $U(1)$-equivariant identification of $C(SU_q(2))$ with the pullback $U(1)$-C*-algebra in the 
following diagram \cite[(3.28)]{dhhmw12}:
\begin{equation}\label{eq:pullsuq2}
\begin{tikzpicture}[xscale=1.2,inner sep=4pt]

\node (A) at (2,3.6) {$C(SU_q(2))$};
\node (B) at (0,1.8) {$C(S^1)$};
\node (C) at (4,1.8) {$\mathcal{T}\otimes C(S^1)^\bullet$};
\node (D) at (2,0) {$C(S^1)\otimes C(S^1)^\bullet$};

\path[->,font=\scriptsize,inner sep=2pt] (A) edge node[above left] {$p_L$} (B)
	(A) edge node[above right] {$p_R$} (C)
	(B) edge node[below left] {$\Delta$} (D)
	(C) edge node[below right] {$\sigma\otimes\id$} (D);
			
\end{tikzpicture} .
\end{equation}
Here, $\Delta$ is the comultiplication of $C(S^1)$, and the maps $p_L$ and $p_R$ are given by:
\begin{equation}\label{eq:p}
\setlength{\arraycolsep}{1pt}
\begin{array}{rlrl}
p_L(\alpha)&=u^*, & p_L(\gamma)&=0 , \\[5pt]
p_R(\alpha)&=\rho(\alpha)\otimes u^*, \qquad\quad & p_R(\gamma)&=\rho(\gamma)\otimes u^* .
\end{array}
\end{equation}
Here $\rho(\alpha)$ and $\rho(\gamma)$ are the elements of Toeplitz algebra that after applying the standard unilateral-shift representation on $\ell^2(\mathbb{N})$ become, respectively, the following operators:
\begin{equation}
\rho(\alpha)e_n=\sqrt{1-q^{2n}}\,e_{n-1}, \qquad\rho(\gamma)e_n=q^ne_n .
\end{equation}
Here $\{e_n\}_{n\geq 0}$ is the standard orthonormal basis of $\ell^2(\N)$.
Note that $\sigma(\rho(\alpha))=u^*$ and $\sigma(\rho(\gamma))=0$.
To end with, observe that the above pullback presentation of $C(SU_q(2))$ proves that the isomorphism class of this C*-algebra does not depend on $q$.

\subsection{The Calow--Matthes quantum sphere $S^3_H$}
The C*-algebra $C(S^3_H)$~\cite{cm00} is the universal C*-algebra generated by two commuting isometries $s_1$ and $s_2$ subject to the relation $(1-s_1s_1^*)(1-s_2s_2^*)=0$. The obvious $U(1)$-action is given on generators by $\alpha_\lambda(s_i)=\lambda s_i\;\forall\;\lambda\in U(1)$.
It was shown in \cite{baum2005k} that this C*-algebra can be equivariantly identified with a certain pullback C*-algebra. After applying the gauge trick as in \eqref{eq:gaugetrick} to the pullback diagram in \cite{baum2005k}, we obtain the following pullback diagram of $U(1)$-C*-algebras:
\begin{equation}\label{eq:S3qS3H}
\begin{tikzpicture}[xscale=1.2,inner sep=4pt]

\node (A) at (2,3.6) {$C(S^3_H)$};
\node (B) at (0,1.8) {$C(S^1)^\bullet\otimes\mathcal{T}^\bullet$};
\node (C) at (4,1.8) {$\mathcal{T}\otimes C(S^1)^\bullet$};
\node (D) at (2,0) {$C(S^1)\otimes C(S^1)^\bullet$};

\path[->,font=\scriptsize] (A) edge (B)
	(A) edge (C)
	(B) edge node[left=2pt] {$\kappa\circ (\id\otimes\sigma)$} (D)
	(C) edge node[right=2pt] {$\sigma\otimes\id$} (D);
			
\end{tikzpicture} .
\end{equation}

The pullback presentations of 
$C(SU_q(2))$ and $C(S^3_H)$ yield the commutative diagram
\begin{equation}\label{ProjectiveMapsOmegaSmall}
\begin{tikzpicture}[scale=1.8]

\node (A) at (1,2) {$C(SU_q(2))$};
\node (B) at (0,1) {$C(S^1)$};
\node (C) at (2,1) {$\mathcal{T}\otimes C(S^1)^\bullet$};
\node (D) at (1,0) {$C(S^1)\otimes C(S^1)^\bullet$};
\node (E) at (5.5,2) {$C(S^3_H)$};
\node (F) at (4.5,1) {$C(S^1)^\bullet\otimes\mathcal{T}^\bullet$};
\node (G) at (6.5,1) {$\mathcal{T}\otimes C(S^1)^\bullet$};
\node (H) at (5.5,0) {$C(S^1)\otimes C(S^1)^\bullet$};

\path[->,font=\scriptsize,inner sep=5pt]
	(A) edge (B)
	(A) edge (C)
	(B) edge node[left=2pt] {$\Delta$} (D)
	(C) edge node[pos=0.7,right] {$\sigma\otimes\id$} (D)
	(E) edge (F)
	(E) edge (G)
	(F) edge node[left] {$\kappa\circ(\id\otimes\sigma)$} (H)
	(G) edge node[right=-1pt] {$\sigma\otimes\id$} (H)
	(A) edge[bend left=10,dashed] node[above] {$\nu$} (E)
	(B) edge[bend right=15] node[below,pos=0.7] {$\id\otimes 1_{\mathcal{T}}$} (F)
	(C) edge[bend left=15] node[above,pos=0.4] {$\id$} (G)
	(D) edge[bend right=10] node[below] {$\id$} (H);
			
\end{tikzpicture} .
\end{equation}
Explicitly,
\begin{equation}\label{eq:nu226}
\nu(v,t\otimes v'):=(v\otimes 1_{\mathcal{T}},t\otimes v') ,
\qquad \forall\;v,v'\in C(S^1),\; t\in\mathcal{T}.
\end{equation}

\subsection{The Hong--Szyma{\'n}ski quantum 4-ball $B^4_q$}

The C*-algebra of the quantum ball $C(B^4_q)$ is defined in \cite{hong02} as the quantum double suspension $\Sigma^2\mathcal{T}$ of the C*-algebra of the quantum disk. In general, the quantum double suspension $\Sigma^2A$ of a C*-algebra $A$ is isomorphic to the C*-subalgebra of $\mathcal{T}\otimes A$ generated by $s\otimes 1$ and 
$(1-ss^*)\otimes A$, where $s$ here is the isometry generating the Toeplitz algebra \cite[Proposition~2.1]{hong08}.
Thus, $C(B^4_q)$ is isomorphic to the C*-subalgebra of $\mathcal{T}\otimes\mathcal{T}$ generated by $s\otimes 1$ and $(1-ss^*)\otimes s$.
The $U(1)$-action on $C(B^4_q)$ is induced by the diagonal action on $\mathcal{T}\otimes\mathcal{T}$.
The *-homomorphism
\begin{equation}\label{eq:varpi}
\mathcal{T}\otimes\mathcal{T}\longrightarrow C(S^1)\times\big(\mathcal{T}\otimes C(S^1)\big) , \qquad\quad
t\otimes t' \longmapsto \big(\sigma(t)\sigma(t'),t_{(0)}\otimes t_{(1)}\sigma(t')\big) ,
\end{equation}
maps the two generators of $C(B^4_q)$ into $C(SU_q(2))$.
Hence, it restricts to a *-homomorphism
\begin{equation}
\varpi:C(B^4_q)\longrightarrow C(SU_q(2)) .
\end{equation}
One easily checks that this *-homomorphism is $U(1)$-equivariant.

We shall also need the *-homomorphism
\begin{equation}
\omega:\mathcal{T}\otimes\mathcal{T}\longrightarrow C(S^3_H), \qquad\quad
\omega(t\otimes t'):=\big(\sigma(t)\otimes t',t_{(0)}\otimes t_{(1)}\sigma(t')\big) ,
\end{equation}
which is clearly $U(1)$-equivariant for the diagonal action on $\mathcal{T}\otimes\mathcal{T}$.
Now we construct a commutative diagram of $U(1)$-C*-algebras:
\begin{equation}\label{eq:B4q}
\begin{tikzpicture}[inner sep=4pt]

\node (A) at (2,3.6) {$C(B^4_q)$};
\node (B) at (0,1.8) {$C(SU_q(2))$};
\node (C) at (4,1.8) {$\mathcal{T}^\bullet\otimes\mathcal{T}^\bullet$};
\node (D) at (2,0) {$C(S^3_H)$};

\path[->,font=\scriptsize,inner sep=2pt]
	(A) edge node[above left] {$\varpi$} (B)
	(A) edge node[above right] {$\iota$} (C)
	(B) edge node[below left] {$\nu$} (D)
	(C) edge node[below right] {$\omega$} (D);
			
\end{tikzpicture} .
\end{equation}
Here, $\iota$ denotes the inclusion of $C(B^4_q)$ into $\mathcal{T}\otimes\mathcal{T}$.
Note that
$\ker\omega=\mathcal{K}\otimes\mathcal{K}$. If $a\otimes b\in C(B^4_q)$ with $a,b\in\mathcal{K}$, clearly $\varpi(a\otimes b)=0$.
Hence, $\ker\omega\subseteq \ker\varpi$.
Furthermore, since $\omega$ is surjective and $\nu$ is injective, the latter inclusion proves that \eqref{eq:B4q} is a pullback diagram \cite[Proposition~3.1]{pedersen1999}.

\subsection{The Vaksman--Soibelman quantum 5-sphere $S_q^5$} 
Odd-dimensional quantum spheres where introduced 
in \cite{VS91}. Here, we are only interested in the quantum 5-sphere~$S^5_q$.
It follows from \cite[(4.2)]{arici2022} that the $U(1)$-C*-algebra $C(S_q^5)$ can be equivariantly identified with the pullback $U(1)$-C*-algebra in the following diagram:
\begin{equation}\label{eq:diagS5q}
\begin{tikzpicture}[inner sep=4pt,xscale=1.2]

\node (A) at (2,3.6) {$C(S^5_q)$};
\node (B) at (0,1.8) {$C(SU_q(2))$};
\node (C) at (4,1.8) {$C(B^4_q)\otimes C(S^1)^\bullet$};
\node (D) at (2,0) {$C(SU_q(2))\otimes C(S^1)^\bullet$};

\path[->,font=\scriptsize,inner sep=1pt]
	(A) edge (B)
	(A) edge (C)
	(B) edge node[below left] {$\delta_R$} (D)
	(C) edge node[below right] {$\varpi\otimes\id$} (D);
			
\end{tikzpicture} .
\end{equation}
Indeed, we can start with the analogous diagram with the diagonal $U(1)$-actions on the tensor products, and then use the gauge trick \eqref{eq:gaugetrick} to transform it to \eqref{eq:diagS5q}.

\subsection{The Heegaard quantum sphere $S^5_H$} The C*-algebra $C(S^5_H)$ \cite{r-j12,hr17,hnpsz} can be equivariantly identified with the pullback $U(1)$-C*-algebra in the diagram
\begin{equation}\label{eq:S5Hpullback}
\begin{tikzpicture}[inner sep=4pt,xscale=1.2]

\node (A) at (2,3.6) {$C(S^5_H)$};
\node (B) at (0,1.8) {$C(S^3_H)^\bullet\otimes \mathcal{T}^\bullet$};
\node (C) at (4,1.8) {$\mathcal{T}\otimes\mathcal{T}\otimes C(S^1)^\bullet$};
\node (D) at (2,0) {$C(S^3_H)\otimes C(S^1)^\bullet$};

\path[->,font=\scriptsize,inner sep=1pt]
	(A) edge (B)
	(A) edge (C)
	(B) edge node[below left] {$\kappa\circ (\id \otimes \sigma)$} (D)
	(C) edge node[below right] {$\omega\otimes\id$} (D);
			
\end{tikzpicture} .
\end{equation}
There is  a $U(1)$-equivariant *-homomorphism
$f:C(S^5_q)\rightarrow C(S^5_H)$ rendering the below diagram commutative:
\begin{equation}\label{eq-large-diag}
\begin{tikzpicture}[scale=1.8]

\node (A) at (1,2) {$C(S^5_q)$};
\node (B) at (0,1) {$C(SU_q(2))$};
\node (C) at (2,1) {$C(B^4_q)\otimes C(S^1)^\bullet$};
\node (D) at (1,0) {$C(SU_q(2))\otimes C(S^1)^\bullet$};
\node (E) at (5.5,2) {$C(S^5_H)$};
\node (F) at (4.5,1) {$C(S^3_H)^\bullet\otimes\mathcal{T}^\bullet$};
\node (G) at (6.5,1) {$\mathcal{T}\otimes \mathcal{T}\otimes C(S^1)^\bullet$};
\node (H) at (5.5,0) {$C(S^3_H)\otimes C(S^1)^\bullet$};

\path[->,font=\scriptsize,inner sep=5pt]
	(A) edge (B)
	(A) edge (C)
	(B) edge node[left=2pt] {$\delta_R$} (D)
	(C) edge node[pos=0.7,right] {$\varpi\otimes\id$} (D)
	(E) edge (F)
	(E) edge (G)
	(F) edge node[left,pos=0.6] {$\kappa\circ(\id\otimes\sigma)$} (H)
	(G) edge node[right=-1pt] {$\omega\otimes\id$} (H)
	(A) edge[bend left=10,dashed] node[above] {$f$} (E)
	(B) edge[bend right=15] node[below,pos=0.7] {$\nu\otimes 1_{\mathcal{T}}$} (F)
	(C) edge[bend left=15] node[above,pos=0.4] {$\iota\otimes\id$} (G)
	(D) edge[bend right=10] node[below] {$\nu\otimes\id$} (H);
			
\end{tikzpicture} .
\end{equation}
It is straightforward to verify that the two sub-diagrams involving the bottom arrow are commutative. Hence, the universal property of pullbacks implies the 
existence and uniqueness of $f$. Explicitly,
\begin{equation}\label{eq:fexp}
f(a,b\otimes v):=\big(\nu(a)\otimes 1_{\mathcal{T}},\iota(b)\otimes v\big)
\end{equation}
for all $a\in C(SU_q(2))$, $b\in C(B^4_q)$ and $v\in C(S^1)$.

\subsection{The Vaksman--Soibelman quantum complex projective plane $\mathbb{C}P^2_q$}
Complex quantum projective spaces coming from Vaksman--Soibelman spheres were introduced by Meyer in \cite{meyer95}.
In particular, one defines $C(\C P^2_q)$ as the $U(1)$-fixed point subalgebra of $C(S^5_q)$.

Remembering that any pullback diagram of $U(1)$-C*-algebras induces a pullback diagram of fixed-point algebras, from the equivariant pullback diagram \eqref{eq:diagS5q} we obtain the pullback diagram
\begin{equation}\label{eq:CP2qpullback}
\begin{tikzpicture}[inner sep=4pt,xscale=1.2]

\node (A) at (2,3.6) {$C(\C P^2_q)$};
\node (B) at (0,1.8) {$C(S^2_q)$};
\node (C) at (4,1.8) {$C(B^4_q)$};
\node (D) at (2,0) {$C(SU_q(2))$};

\path[->,font=\scriptsize,inner sep=3pt]
	(A) edge (B)
	(A) edge (C)
	(B) edge node[above right] {$\subseteq$} (D)
	(C) edge node[above left] {$\varpi$} (D);
			
\end{tikzpicture} .
\end{equation}
Here $C(S^2_q):=C(SU_q(2))^{U(1)}\xrightarrow{\;\;\subseteq\;\;}C(SU_q(2))$ denotes the inclusion of the C*-algebra of the standard Podle\'s sphere.

\subsection{The multipushout quantum complex projective plane $\mathbb{C}P^2_H$}
The C*-algebra $C(\mathbb{C}P^2_H)$ was originally defined as a multipullback, and can be equivalently presented as an iterated pullback \cite{r-j12}. (This was later generalized tp an arbitrary dimension in \cite{hkz12}.)
Here, following \cite{hr17}, we shall identify $C(\mathbb{C}P^2_H)$ with $C(S^5_H)^{U(1)}$.
As in the preceeding section, from the $U(1)$-equivariant pullback diagram \eqref{eq:S5Hpullback}, we obtain the pullback diagram
\begin{equation}\label{eq:CP2Hpullback}
\begin{tikzpicture}[xscale=2,yscale=1.8]

\node (E) at (5.5,2) {$C(\C P^2_H)$};
\node (F) at (4.5,1) {$\big(C(S^3_H)^\bullet\otimes\mathcal{T}^\bullet\big)^{U(1)}$};
\node (G) at (6.5,1) {$\mathcal{T}\otimes \mathcal{T}$};
\node (H) at (5.5,0) {$C(S^3_H)$};

\path[->,font=\scriptsize,inner sep=2pt]
	(E) edge (F)
	(E) edge (G)
	(F) edge node[below left,yshift=2pt] {$\big(\kappa\circ(\id\otimes\sigma)\big){}^{U(1)}$} (H)
	(G) edge node[below right] {$\omega$} (H);
			
\end{tikzpicture}.
\end{equation}
Here $\big(\kappa\circ(\id\otimes\sigma)\big){}^{U(1)}$ is the map between fixed-point algebras
induced by the map $\kappa\circ(\id\otimes\sigma)$ in \eqref{eq:S5Hpullback}.
The $U(1)$-equivariant diagram \eqref{eq-large-diag} induces the diagram
\begin{equation}\label{ProjectiveMaps}
\begin{tikzpicture}[xscale=2,yscale=1.8,baseline={([yshift=-6pt]current bounding box.center)}]

\node (A) at (1,2) {$C(\C P^2_q)$};
\node (B) at (0,1) {$C(S^2_q)$};
\node (C) at (2,1) {$C(B^4_q)$};
\node (D) at (1,0) {$C(SU_q(2))$};
\node (E) at (5.5,2) {$C(\C P^2_H)$};
\node (F) at (4.5,1) {$\big(C(S^3_H)^\bullet\otimes\mathcal{T}^\bullet\big)^{U(1)}$};
\node (G) at (6.5,1) {$\mathcal{T}\otimes \mathcal{T}$};
\node (H) at (5.5,0) {$C(S^3_H)$};

\path[->,font=\scriptsize,inner sep=5pt]
	(A) edge (B)
	(A) edge (C)
	(B) edge node[left=2pt] {$\subseteq$} (D)
	(C) edge node[pos=0.7,right] {$\varpi$} (D)
	(E) edge (F)
	(E) edge (G)
	(F) edge node[left=2pt,pos=0.6] {$\big(\kappa\circ(\id\otimes\sigma)\big)^{U(1)}$} (H)
	(G) edge node[right=-1pt] {$\omega$} (H)
	(A) edge[bend left=10,dashed] node[above] {$f^{U(1)}$} (E)
	(B) edge[bend right=15] node[below,pos=0.7] {$(\nu\otimes 1_{\mathcal{T}})^{U(1)}$} (F)
	(C) edge[bend left=15] node[above,pos=0.4] {$\iota$} (G)
	(D) edge[bend right=10] node[below] {$\nu$} (H);
			
\end{tikzpicture}.
\end{equation}

\section{Associated-module construction of a Milnor module}
\noindent
Recall first that, for any integer $n$ and for any unital C*-algebra $A$ equipped 
with a $U(1)$-action $\alpha: U(1)\to\mathrm{Aut}(A)$,
we can define the $n$-th spectral subspace (the section module of the associated line bundle with winding number $n$)
as
\begin{equation}\label{spectral}
A_k:=\{a\in A\;|\;\alpha_\lambda(a)=\lambda^ka\}.
\end{equation}
In particular, we can take $A=C(S^{2n+1}_q)$ and $A=C(S^{2n+1}_H)$, and denote the respective spectral subspaces by $\tilde{L}_k$
and $L_k$. 
As was shown in~\cite{hnpsz}, the C*-algebra $C(\mathbb{C}P_H^n)$ of the multipullback
quantum complex projective space $\mathbb{C}P_H^n$ can be realized as a fixed-point subalgebra
for the diagonal $U(1)$-action on the C*-algebra $C(S^{2n+1}_H)$ of the Heegaard quantum 
sphere~$S^{2n+1}_H$. 

Next, we  write $P_1:=(C(S^3_H)^\bullet\otimes\mathcal{T}^\bullet)^{U(1)}$ to shorten the notation.
It was proven in \cite{r-j12} that
\begin{equation}
\label{BasicSplitEq}
\begin{gathered}
K_0(C(\C P^2_H))) =\mathrm{s}(K_0(P_1))\oplus\partial_{10}(K_1(C(S_H^3))) , \\
K_0(P_1) \cong\mathbb{Z}\oplus\mathbb{Z} , \\
\partial_{10}(K_1(C(S_H^3))) \cong\mathbb{Z} .
\end{gathered}
\end{equation}
Here $\partial_{10}$ is the Milnor connecting homomorphism of the diagram \eqref{eq:CP2Hpullback}, and $\mathrm{s}$ is a splitting of the epimorphism 
$K_0(C(\C P^2_H)))\to K_0(P_1)$
induced by \eqref{eq:CP2Hpullback}. Moreover, it was shown in \cite{hr17} that the splitting can be chosen in such a way that 
its image $\mathrm{s}(K_0(P_1))$ is generated by $[1]$ and $[L_1]-[1]$.

The goal of this section is to identify a generator of 
$\partial_{10}(K_1(C(S_H^3)))$, which must be of the form 
$\partial_{10}([a])$ for some invertible matrix~$a\in GL(n,C(S^3_H))$. 
To this end, using again the Mayer--Vietoris six-term exact sequence in K-theory,
first we show  that $K_0(C(\C P^2_q))$ enjoys an analogous decomposition 
as~$K_0(C(\C P^2_H))$:
\begin{equation}
\label{BasicSplitEq2}
K_0(C(\C P^2_q))=\mathrm{t}(K_0(C(S_q^2)))\oplus\partial_{10}(K_1(C(SU_q(2)))).
\end{equation}
Here, $\partial_{10}$ is the Milnor connecting homomorphism of the diagram \eqref{eq:CP2qpullback}, 
and $\mathrm{t}$ is a splitting of the epimorphism $K_0(C(\C P^2_q)))\to K_0(C(S^2_q))$
induced by \eqref{eq:CP2qpullback}.
Next, using \cite{h-pm00} and \cite[Theorem~5.1]{hnpsz} 
(cf.~\cite[Theorem~1.1]{hm}), we can easily see that
$\mathrm{t}$ can be chosen in such a way that 
$[1]$ and $[\tilde{L}_1]-[1]$ generate $\mathrm{t}(K_0(C(S^2_q)))$. 
Moreover, we will show that $[\tilde{L}_1\oplus\tilde{L}_{-1}]-2[1]$ generates $\partial_{10}(K_1(C(SU_q(2))))$.

Finally, using again \cite[Theorem~5.1]{hnpsz}, we employ 
the $U(1)$-equivariant *-homomorphism $f:C(S^5_q)\rightarrow C(S^5_H)$ in \eqref{eq:fexp} to conclude that
\[
(f^{U(1)})_*:K_0(C(\C P^2_q))\longrightarrow K_0(C(\C P^2_H))
\]
maps $[\tilde{L}_n]$ to $[L_n]$ for any $n\in\mathbb{Z}$.
Then we prove that
$(f^{U(1)})_*$ is an isomorphism of $K$-groups. This allows us to infer that
$[L_1\oplus L_{-1}]-2[1]$ generates $\partial_{10}(K_1(C(S^3_H)))$.

\subsection{Modules associated to piecewise cleft coactions}

In this section,  we will unravel the structure of  Milnor's module of Section~\ref{Milnor} 
in the setting of noncommutative  vector bundles associated to compact quantum principal bundles 
obtained by glueing two cleft parts. To this end, we take a Hopf algebra $\mathcal{H}$ 
and consider the following pullback diagram of $\mathcal{H}$-comodule algebras:
\begin{equation}\label{eq:319}
\begin{tikzpicture}[xscale=1.6,yscale=1.5]

\node (A) at (1,2) {$\mathcal{P}$};
\node (B) at (0,1) {$\mathcal{P}_1$};
\node (C) at (2,1) {$\mathcal{P}_2$};
\node (D) at (1,0) {$\mathcal{P}_{12}$};

\path[->,font=\scriptsize,inner sep=1pt] (A) edge (B)
	(A) edge (C)
	(B) edge node[below left] {$\pi_1$} (D)
	(C) edge node[below right] {$\,\pi_2$} (D);
			
\end{tikzpicture} .
\end{equation}
Assume next that
both $\mathcal{P}_1$ and $\mathcal{P}_2$ are \emph{cleft}, i.e.~there exist
unital $\mathcal{H}$-colinear convolution-invertible  maps
\begin{equation}
\gamma_i:\mathcal{H}\longrightarrow\mathcal{P}_i,\quad i=1,2\,,
\end{equation}
called \emph{cleaving maps}.
Here the \emph{convolution product} of linear maps $f,g\colon C\to A$ from a coalgebra $(C,\Delta,\varepsilon)$ to an algebra $(A,m,\eta)$
is defined to be $f\ast g:=m\circ(f\otimes g)\circ\Delta$ with the neutral element being the composition $\eta\circ\varepsilon$ of the counit map
with the unit map.

Furthermore, let
$\pi^{\co\mathcal{H}}_i:\mathcal{P}_i^{\co \mathcal{H}}\longrightarrow \mathcal{P}_{12}^{\co \mathcal{H}}$  be 
the map induced by $\pi_i$, $i=1,2$, and let $V$ be a finite-dimensional vector space that is also a left $\mathcal{H}$-comodule.
Consider now the following isomorphisms
of, respectively, $\mathcal{P}_i^{\co \mathcal{H}}$-modules and $\mathcal{P}_{12}^{\co \mathcal{H}}$-modules given by the cleaving maps:
\begin{equation}\label{lambdaiso}
\setlength{\arraycolsep}{1pt}
\begin{array}{rlrl}
\Lambda_i &:\mathcal{P}_i^{\co \mathcal{H}}\otimes V\longrightarrow \mathcal{P}_i\mathbin{\Box}^{\mathcal{H}} V,
 \quad
& b_i\otimes v &\longmapsto b_i\gamma_i(v\sw{-1})\otimes v\sw{0},
\\[5pt]
\Lambda^i_{12} &:\mathcal{P}_{12}^{\co \mathcal{H}}\otimes V\longrightarrow \mathcal{P}_{12}\mathbin{\Box}^{\mathcal{H}} V,
& b\otimes v &\longmapsto b(\pi_i\circ\gamma_i)(v\sw{-1})\otimes v\sw{0}.
\end{array}
\end{equation}
Note that the fact that $\pi_i$ is an $\mathcal{H}$-colinear unital algebra homomorphism ensures that $\pi_i\circ\gamma_i$ is a cleaving map.
Observe also that the inverses of these isomorphisms are given by the convolution inverses of cleaving maps.
The isomorphisms in \eqref{lambdaiso} give rise to the following commutative diagram:
\begin{equation}\label{MilnorAssocDiag}
\begin{tikzpicture}[scale=2.5,->]

\node (a1) at (0,2) {$\mathcal{P}_1^{\co \mathcal{H}}\otimes V$};
\node (a2) at (3,2) {$\mathcal{P}_2^{\co \mathcal{H}}\otimes V$};
\node (a3) at ($(a1)+(45:1.3)$) {$\mathcal{P}_1\mathbin{\Box}^{\mathcal{H}} V$};
\node (a4) at ($(a2)+(45:1.3)$) {$\mathcal{P}_2\mathbin{\Box}^{\mathcal{H}} V$};
\node (b1) at (0,0) {$\mathcal{P}_{12}^{\co \mathcal{H}}\otimes V$};
\node (b2) at (3,0) {$\mathcal{P}_{12}^{\co \mathcal{H}}\otimes V$};
\node (b3) at ($(b1)+(45:1.3)$) {$\mathcal{P}_{12}\mathbin{\Box}^{\mathcal{H}} V$};
\node (b4) at ($(b2)+(45:1.3)$) {$\mathcal{P}_{12}\mathbin{\Box}^{\mathcal{H}} V$};

\begin{scope}[font=\scriptsize]

\path (a1) edge node[left] {$\pi^{\co\mathcal{H}}_1\otimes\id$} (b1);
\path (a3) edge node[right] {$\pi_1\otimes\id$} (b3);
\path (a4) edge node[right] {$\pi_2\otimes\id$} (b4);
\path (b1) edge[dashed] node[below=2pt] {$\chi$} (b2);
\path (b3) edge node[pos=0.4,above=2pt] {$=$} (b4);
\node[fill=white] at ($(b3)!0.695!(b4)$) {\;\;};
\path (a2) edge node[pos=0.3,left] {$\pi^{\co\mathcal{H}}_2\otimes\id$} (b2);
\path (a1) edge node[above left] {$\Lambda_1$} (a3);
\path (a2) edge node[above left] {$\Lambda_2$} (a4);
\path (b1) edge node[below right] {$\Lambda^1_{12}$} (b3);
\path (b2) edge node[below right] {$\Lambda^2_{12}$} (b4);

\end{scope}

\end{tikzpicture}.
\end{equation}
Here $\chi:=( \Lambda^2_{12})^{-1}\circ \Lambda^1_{12}$. As $\chi$ is an isomorphism of $\mathcal{P}_{12}^{\co \mathcal{H}}$-modules,
it suffices to compute it on $1\otimes v$:
\begin{align}\label{GammaTheorem}
\chi (1\otimes v)&=( \Lambda^2_{12})^{-1}\big((\pi_1\circ\gamma_1)(v_{(-1)})\otimes v_{(0)}\big)
\nonumber\\
&=(\pi_1\circ\gamma_1)(v_{(-2)})(\pi_2\circ\gamma_2^{-1})(v_{(-1)})\otimes v_{(0)}
\nonumber\\
&=
\big((\pi_1\circ\gamma_1)\ast(\pi_2\circ\gamma_2^{-1})\big)(v_{(-1)})\otimes v_{(0)}\,.
\end{align}

Next, we apply the Peter--Weyl functor of \cite{bch17} to use \eqref{GammaTheorem} in the context of C*-algebras.
Let $(H,\Delta)$ be a compact quantum group and let
\begin{equation}\label{EqMiln2}
\begin{tikzpicture}[xscale=1.6,yscale=1.5]

\node (A) at (1,2) {$A$};
\node (B) at (0,1) {$A_1$};
\node (C) at (2,1) {$A_2$};
\node (D) at (1,0) {$A_{12}$};

\path[->,font=\scriptsize] (A) edge (B)
	(A) edge (C)
	(B) edge (D)
	(C) edge (D);
			
\end{tikzpicture}
\end{equation}
be a pullback diagram in the category of unital $H$-C*-algebras.
Denote by $\mathcal{H}$  Woronowicz's Peter--Weyl *-Hopf algebra generated 
by the matrix elements of  finite-dimensional  representations of $(H,\Delta)$.
Furthermore, recall that, for any unital C*-algebra $A$ endowed with a coaction
$\delta_R:A\rightarrow A\otimes_{\min} H$, the Peter--Weyl $\mathcal{H}$-comodule algebra 
\cite[(0.4)]{bch17}
consists of all 
$a\in A$ such that $\delta_R(a)\in A\otimes\mathcal{H}$. (Observe that here it is crucial that the tensor product is algebraic.)
Next, by $\mathcal{P}_H(A)$, $\mathcal{P}_H(A_1)$, $\mathcal{P}_H(A_2)$ and $\mathcal{P}_H(A_{12})$ 
denote the respective Peter--Weyl $\mathcal{H}$-comodule algebras of $H$-C*-algebras $A_1$, $A_2$, and
$A_{12}$. 
Now, to abbreviate notation, put $B:=\mathcal{P}_H(A)^{co\mathcal{H}}$, $B_1:=\mathcal{P}_H(A_1)^{co\mathcal{H}}$, 
$B_2:=\mathcal{P}_H(A_2)^{co\mathcal{H}}$ and 
$B_{12}:=\mathcal{P}_H(A_{12})^{co\mathcal{H}}$
for the respective fixed-point C*-algebras.
Recall also that the Peter--Weyl functor
transforms pullback diagrams into pullback
diagrams \cite[Lemma~2]{bh14}, so applying it to the pullback diagram \eqref{EqMiln2} yields the pullback diagram~\eqref{eq:319}
of Peter--Weyl comodule algebras
\begin{equation}\label{pwpull}
\begin{tikzpicture}[xscale=1.6,yscale=1.5]

\node (A) at (1,2) {$\mathcal{P}_H(A)$};
\node (B) at (0,1) {$\mathcal{P}_H(A_1)$};
\node (C) at (2,1) {$\mathcal{P}_H(A_2)$};
\node (D) at (1,0) {$\mathcal{P}_H(A_{12})$};

\path[->,font=\scriptsize,inner sep=1pt] (A) edge (B)
	(A) edge (C)
	(B) edge node[below left] {$\pi_1$} (D)
	(C) edge node[below right] {$\,\pi_2$} (D);
			
\end{tikzpicture} 
\end{equation}
and the pullback diagram of fixed-point C*-algebras
\begin{equation}\label{pwpullfixed}
\begin{tikzpicture}[xscale=1.6,yscale=1.5]

\node (A) at (1,2) {$B$};
\node (B) at (0,1) {$B_1$};
\node (C) at (2,1) {$B_2$};
\node (D) at (1,0) {$B_{12}$};

\path[->,font=\scriptsize,inner sep=1pt] (A) edge (B)
	(A) edge (C)
	(B) edge node[below left] {$\pi_1^{co\mathcal{H}}$} (D)
	(C) edge node[below right] {$\,\pi_2^{co\mathcal{H}}$} (D);
			
\end{tikzpicture} .
\end{equation}
Therefore, remembering that  the surjectivity
of $A_2\to A_{12}$ implies the surjectivity of $B_2\to B_{12}$, 
we can combine \eqref{GammaTheorem} with Theorem~\ref{thm-Milnor} to obtain:

\begin{thm}\label{thm41}
Let $(H,\Delta)$ be a compact quantum group, let \eqref{EqMiln2} be a pullback diagram in the category of unital $H$-C*-algebras
with $A_2\to A_{12}$ surjective and such that both Peter--Weyl comodule algebras $\mathcal{P}_H(A_1)$ and $\mathcal{P}_H(A_2)$
are cleft with cleaving maps $\gamma_1$ and~$\gamma_2$, respectively. Also, let
$V$ be an $n$-dimensional left comodule over the Woronowicz--Peter--Weyl Hopf algebra $\mathcal{H}$ of~$(H,\Delta)$. Then
\begin{gather*}
\mathcal{P}_H(A)\Box^\mathcal{H}V\cong B^{2n}p_a.
\end{gather*}
Here $p_a$ is the Milnor projection \eqref{MilnProjEq}, $a\in GL_n(B_{12})$, and
$$
a_{kl}:=\big((\pi_1\circ\gamma_1)\ast(\pi_2\circ\gamma_2^{-1})\big)(u_{kl}),\quad
 \sum_{l=1}^n u_{kl}\otimes e_l:=(e_k)_{(-1)}\otimes (e_k)_{(0)},
 $$
 for a basis $\{e_k\}_1^n$ of~$V$.
\end{thm}

\subsection{Associated-module construction in the Vaksman--Soibelman case}
\label{GroupIsomSect}
Let $A$ be a unital $U(1)$-C*-algebra. Then $\mathcal{P}_{U(1)}(A)$ stands for its
$\mathcal{O}(U(1))$-Peter--Weyl-comodule algebra. 
By applying the Peter--Weyl functor to the pullback diagram \eqref{eq:diagS5q}, we get the pullback diagram
of $\mathcal{O}(U(1))$-comodule *-algebras:
\begin{equation}
\begin{tikzpicture}[inner sep=4pt,xscale=1.2,baseline={([yshift=-4pt]current bounding box.center)}]

\node (A) at (2,3.6) {$\mathcal{P}_{U(1)}(C(S^5_q))$};
\node (B) at (0,1.8) {$\mathcal{P}_{U(1)}(C(SU_q(2)))$};
\node (C) at (4,1.8) {$C(B^4_q)\otimes\mathcal{O}(S^1)^\bullet$};
\node (D) at (2,0) {$C(SU_q(2))\otimes\mathcal{O}(S^1)^\bullet$};

\path[->,font=\scriptsize,inner sep=1pt]
	(A) edge (B)
	(A) edge (C)
	(B) edge node[below left] {$\delta_R$} (D)
	(C) edge node[below right] {$\varpi\otimes\id$} (D);
			
\end{tikzpicture}.
\end{equation}
Here we obtain a restriction of the coaction map $\delta_R$ but use the same symbol to signify the restriction.

Applying now the prolongation functor
$(\,\_\,)\mathbin{\Box}^{\mathcal{O}(U(1))}\mathcal{O}(SU_q(2))$ to the above diagram, we obtain the pullback diagram of $\mathcal{O}(SU_q(2))$-comodule algebras:
\begin{equation}\label{eq:46}
\begin{tikzpicture}[inner sep=4pt,yscale=2,xscale=3.5,baseline={([yshift=-4pt]current bounding box.center)}]

\node (A) at (1,2) {$\mathcal{P}_{U(1)}(C(S^5_q))\mathbin{\Box}^{\mathcal{O}(U(1))}\mathcal{O}(SU_q(2))^\bullet$};
\node (B) at (0,1) {$\mathcal{P}_{U(1)}(C(SU_q(2)))\mathbin{\Box}^{\mathcal{O}(U(1))}\mathcal{O}(SU_q(2))^\bullet$};
\node (C) at (2,1) {$C(B^4_q)\otimes\mathcal{O}(SU_q(2))^\bullet$};
\node (D) at (1,0) {$C(SU_q(2))\otimes\mathcal{O}(SU_q(2))^\bullet$};

\path[->,font=\scriptsize,inner sep=1pt]
	(A) edge (B)
	(A) edge (C)
	(B) edge node[below left] {$\subseteq$} (D)
	(C) edge node[below right] {$\varpi\otimes\id$} (D);
			
\end{tikzpicture}.
\end{equation}
In this diagram, the right $\mathcal{O}(SU_q(2))$-coaction is always induced by the coproduct applied to the rightmost factor of (co)tensor products. The cotensor product is constructed with the left coaction of $\mathcal{O}(U(1))$ on $\mathcal{O}(SU_q(2))$ induced by \eqref{eq:resu1}.

\begin{lemma}\label{LemGenTilda}
Let $U$ be the fundamental representation \eqref{funrep} of $SU_q(2)$. Then 
\begin{equation}\label{theform}
\partial_{10}([U])=[\tilde{L}_1\oplus \tilde{L}_{-1}]-2[1].
\end{equation}
In particular, $[\tilde{L}_1\oplus\tilde{L}_{-1}]-2[1]$ generates~$\partial_{10}\big(K_1(C(SU_q(2)))\big)$.
\end{lemma}
\begin{proof}
We apply Theorem \ref{thm41} to \eqref{eq:46} viewed as the diagram \eqref{pwpull} with $H=C(SU_q(2))$.
Now the pullback diagram of fixed-point C*-algebras \eqref{pwpullfixed} becomes~\eqref{eq:CP2qpullback}.
It is easy to see that we have the following cleaving maps:
\begin{equation}
\setlength{\arraycolsep}{1pt}
\begin{array}{rlrl}
\gamma_1 &: \mathcal{O}(SU_q(2))\longrightarrow \mathcal{P}_{U(1)}(SU_q(2))\mathbin{\Box}^{\mathcal{O}(U(1))}\mathcal{O}(SU_q(2)),
\qquad & x &\longmapsto x\sw{1}\otimes x\sw{2},\\[5pt]
\gamma_2 &: \mathcal{O}(SU_q(2))\longrightarrow
C(B^4_q)\otimes \mathcal{O}(SU_q(2)),
& x &\longmapsto 1\otimes x .
\end{array}
\end{equation}
Furthermore, noting that $\pi_1$ is a set inclusion and $\pi_2=\varpi\otimes\id$, we compute:
\begin{equation}\label{IsomMilnEq}
\big((\pi_1\circ\gamma_1)\ast (\pi_2\circ\gamma_2^{-1})\big)(x)=
(x\sw{1}\otimes x\sw{2})(1\otimes S(x\sw{3}))=x .
\end{equation}
Thus, the map  $(\pi_1\circ\gamma_1)\ast (\pi_2\circ\gamma_2^{-1}):\mathcal{H}\to B_{12}$ becomes the inclusion of 
$\mathcal{O}(SU_q(2))$ into $C(SU_q(2))$.

Next, let $V:=\C^2$ be the left $\mathcal{O}(SU_q(2))$-comodule defined via the fundamental representation matrix~$U$ in \eqref{funrep}.
Thus, the left $B$-module $\mathcal{P}_H(A)\Box^\mathcal{H}V$ becomes
\begin{align}
\mathcal{P}_H(A)\Box^\mathcal{H}V &=\big(\mathcal{P}_{U(1)}(C(S^5_q)) \mathbin{\Box}^{\mathcal{O}(U(1))}
\mathcal{O}(SU_q(2))\big)\mathbin{\Box}^{\mathcal{O}(SU_q(2))}\C^2\nonumber\\
&=\mathcal{P}_{U(1)}(C(S^5_q))\mathbin{\Box}^{\mathcal{O}(U(1))}\C^2\nonumber\\
&=\bigl(\mathcal{P}_{U(1)}(C(S^5_q))\mathbin{\Box}^{\mathcal{O}(U(1))}\C^{-}\bigr)\oplus
\bigl(\mathcal{P}_{U(1)}(C(S^5_q))\mathbin{\Box}^{\mathcal{O}(U(1))}\C^{+}\bigr)\nonumber\\
&=\tilde{L}_{-1}\oplus\tilde{L}_{1}\,.
\label{TlTlEq}
\end{align}
Here $\C^\pm$ stands for the left $\mathcal{O}(U(1))$-comodule $\C$ given by the coaction $1\mapsto u^{\pm 1}\otimes 1$.
Finally, applying Theorem~\ref{thm41}, we obtain \eqref{theform}.
Now, to complete the proof, it suffices to note that $[U]$ generates $K_1(C(SU_q(2)))$.
\end{proof}

\subsection{Reducing the multipushout case to the Vaksman--Soibelman case}\label{sec-gaugin}

In this section, we show that the map $f^{U(1)}$ in diagram \eqref{ProjectiveMaps} induces an isomorphim in K-theory. This is the content of next theorem, whose prove requires three lemmas.
We are going to use the K{\"u}nneth theorem \cite[Theorem 23.1.3]{b-b98} and the fact that all our K-theory groups are flat over $\Z$, so the torsion term in the K{\"u}nneth exact sequence disappears.

\begin{lemma}
The map $\iota:C(B^4_q)\to\mathcal{T}\otimes\mathcal{T}$ induces an isomorphism in K-theory.
\end{lemma}

\begin{proof}
It was shown in \cite[Example 6.5]{hong02} that  $K_0(C(B^4_q))\cong\Z$ is generated by $[1]$, and $K_1(C(B^4_q))=0$.
Similarly, by the K{\"u}nneth theorem, $K_0(\mathcal{T}\otimes\mathcal{T})\cong\Z$ is generated by $[1]$, and $K_1(\mathcal{T}\otimes\mathcal{T})=0$.
Finally, since $\iota$ is a unital map, we conclude that $\iota_*$ is an isomorphism between the K-groups.
\end{proof}

\begin{lemma}\label{lemma:nu43}
The map $\nu:C(SU_q(2))\to C(S^3_H)$ in \eqref{eq:nu226} induces an isomorphism in K-theory.
\end{lemma}

\begin{proof}
We want to apply Theorem~\ref{thm:3outof4} to the diagram \eqref{ProjectiveMapsOmegaSmall}. To this end, we need to verify that the three solid horizontal arrows in the diagram \eqref{ProjectiveMapsOmegaSmall} induce isomorphisms in K-theory.
Clearly, the two identity maps induce isomorphisms in K-theory, and
$\id\otimes 1_{\mathcal{T}}$ induces an isomorphism in K-theory by the K{\"u}nneth exact sequence applied to the tensor product $C(S^1)\otimes\mathcal{T}$.
\end{proof}

\begin{lemma}
The map $(\nu\otimes 1_{\mathcal{T}})^{U(1)}:C(S^2_q)\to \big(C(S^3_H)^\bullet\otimes\mathcal{T}^\bullet\big)^{U(1)}$ induces an isomorphism in K-theory.
\end{lemma}

\begin{proof}
We tensor the right pullback diagram in \eqref{ProjectiveMapsOmegaSmall} with $\mathcal{T}$ to obtain the commutative diagram of $U(1)$-C*-algebras:
\begin{equation}\label{eq:413}
\begin{tikzpicture}[scale=2]

\node (A) at (1,2) {$C(SU_q(2))$};
\node (B) at (0,1) {$C(S^1)$};
\node (C) at (2,1) {$\mathcal{T}\otimes C(S^1)^\bullet$};
\node (D) at (1,0) {$C(S^1)\otimes C(S^1)^\bullet$};
\node (E) at (5,2) {$C(S^3_H)^\bullet\otimes\mathcal{T}^\bullet$};
\node (F) at (4,1) {$C(S^1)^\bullet\otimes\mathcal{T}^\bullet\otimes\mathcal{T}^\bullet$};
\node (G) at (6,1) {$\mathcal{T}\otimes C(S^1)^\bullet\otimes\mathcal{T}^\bullet$};
\node (H) at (5,0) {$C(S^1)\otimes C(S^1)^\bullet\otimes\mathcal{T}^\bullet$};

\path[->,font=\scriptsize,inner sep=5pt]
	(A) edge (B)
	(A) edge (C)
	(B) edge (D)
	(C) edge (D)
	(E) edge (F)
	(E) edge (G)
	(F) edge (H)
	(G) edge (H)
	(A) edge[bend left=10,dashed] node[above] {$\nu\otimes 1_{\mathcal{T}}$} (E)
	(B) edge[bend right=15] node[below,pos=0.7] {$\id\otimes 1_{\mathcal{T}\otimes\mathcal{T}}$} (F)
	(C) edge[bend left=15] node[above,pos=0.4] {$\id\otimes 1_{\mathcal{T}}$} (G)
	(D) edge[bend right=10] node[below] {$\id\otimes 1_{\mathcal{T}}$} (H);
			
\end{tikzpicture}
\end{equation}
Herein, the dots indicate the factors in tensor products that contribute to the diagonal $U(1)$-action.
Note that the maps on the pullback squares are omitted as we are not going to need their explicit form. Next, consider the following gauge maps:
\begin{equation}
\setlength{\arraycolsep}{1pt}
\begin{array}{rlrl}
C(S^1)^\bullet\otimes\mathcal{T}^\bullet &\longrightarrow
C(S^1)^\bullet\otimes\mathcal{T} , &
v\otimes t &\longmapsto vt_{(-1)}\otimes t_{(0)} ,
\\[5pt]
C(S^1)^\bullet\otimes\mathcal{T}^\bullet\otimes\mathcal{T}^\bullet &\longrightarrow
C(S^1)^\bullet\otimes\mathcal{T}\otimes\mathcal{T} , \quad &
v\otimes t\otimes t' &\longmapsto vt_{(-1)}t'_{(-1)}\otimes t_{(0)}\otimes t'_{(0)} .
\end{array}
\end{equation}
We apply the first one to the bottom and right nodes of the right square in \eqref{eq:413}, and the second one to the left node of the same square. Since the first gauge map is the identity on the image of $\id\otimes 1_{\mathcal{T}}$, and the second gauge map is the identity on the image of $\id\otimes 1_{\mathcal{T}}\otimes 1_{\mathcal{T}}$, the diagram \eqref{eq:413} becomes
\begin{equation}
\begin{tikzpicture}[scale=2,baseline={([yshift=-2pt]current bounding box.center)}]

\node (A) at (1,2) {$C(SU_q(2))$};
\node (B) at (0,1) {$C(S^1)$};
\node (C) at (2,1) {$\mathcal{T}\otimes C(S^1)^\bullet$};
\node (D) at (1,0) {$C(S^1)\otimes C(S^1)^\bullet$};
\node (E) at (5,2) {$C(S^3_H)^\bullet\otimes\mathcal{T}^\bullet$};
\node (F) at (4,1) {$C(S^1)^\bullet\otimes\mathcal{T}\otimes\mathcal{T}$};
\node (G) at (6,1) {$\mathcal{T}\otimes C(S^1)^\bullet\otimes\mathcal{T}$};
\node (H) at (5,0) {$C(S^1)\otimes C(S^1)^\bullet\otimes\mathcal{T}$};

\path[->,font=\scriptsize,inner sep=5pt]
	(A) edge (B)
	(A) edge (C)
	(B) edge (D)
	(C) edge (D)
	(E) edge (F)
	(E) edge (G)
	(F) edge (H)
	(G) edge (H)
	(A) edge[bend left=10,dashed] node[above] {$\nu\otimes 1_{\mathcal{T}}$} (E)
	(B) edge[bend right=15] node[below,pos=0.7] {$\id\otimes 1_{\mathcal{T}\otimes\mathcal{T}}$} (F)
	(C) edge[bend left=15] node[above,pos=0.4] {$\id\otimes 1_{\mathcal{T}}$} (G)
	(D) edge[bend right=10] node[below] {$\id\otimes 1_{\mathcal{T}}$} (H);
			
\end{tikzpicture} .
\end{equation}
Note that the horizontal maps remains unchanged. Passing to the fixed-point subalgebras, we obtain the commutative diagram
\begin{equation}
\begin{tikzpicture}[scale=1.8,baseline={([yshift=-5pt]current bounding box.center)}]

\node (A) at (1,2) {$C(S^2_q)$};
\node (B) at (0,1) {$\C$};
\node (C) at (2,1) {$\mathcal{T}$};
\node (D) at (1,0) {$C(S^1)$};
\node (E) at (5.5,2) {$\big(C(S^3_H)^\bullet\otimes\mathcal{T}^\bullet\big)^{U(1)}$};
\node (F) at (4.5,1) {$\mathcal{T}\otimes\mathcal{T}$};
\node (G) at (6.5,1) {$\mathcal{T}\otimes\mathcal{T}$};
\node (H) at (5.5,0) {$C(S^1)\otimes\mathcal{T}$};

\path[->,font=\scriptsize,inner sep=5pt]
	(A) edge (B)
	(A) edge (C)
	(B) edge (D)
	(C) edge (D)
	(E) edge (F)
	(E) edge (G)
	(F) edge (H)
	(G) edge (H)
	(A) edge[bend left=10,dashed] node[above] {$(\nu\otimes 1_{\mathcal{T}})^{U(1)}$} (E)
	(B) edge[bend right=15] node[below,pos=0.7] {$ 1_{\mathcal{T}\otimes \mathcal{T}}$} (F)
	(C) edge[bend left=15] node[above,pos=0.4] {$\id\otimes 1_{\mathcal{T}}$} (G)
	(D) edge[bend right=10] node[below] {$\id\otimes 1_{\mathcal{T}}$} (H);
			
\end{tikzpicture} .
\end{equation}
Arguing as in the previous lemmas, we conclude that all three solid
horizontal arrows induce isomorphisms in K-theory.
Therefore, we infer from Theorem~\ref{thm:3outof4} that 
\mbox{$(\nu\otimes 1_{\mathcal{T}})^{U(1)}$} also induces an isomorphism in K-theory.
\end{proof}

Finally, we combine the above three lemmas and apply Theorem~\ref{thm:3outof4} to the diagram \eqref{ProjectiveMaps}
to obtain:

\begin{thm}
\label{KIsomMainThm}
The map $(f^{U(1)})_{*}: K_{*}(C(\C P^2_q))\to K_{*}(C(\C P^2_H))$ is an isomorphism.
\end{thm}

\begin{cor}
\label{MainCorSuff}
The element $[L_1\oplus L_{-1}]-2[1]$ generates $\partial_{10}(K_1(C(S^3_H)))$.
\end{cor}
\begin{proof}
Note first that \cite[Theorem~5.1]{hnpsz}
implies that 
$(f^{U(1)})_*$ maps $[\tilde{L}_n]$ to $[L_n]$ for any $n\in\mathbb{Z}$. If we apply this map to both sides of \eqref{theform}, we get
\begin{equation}
(f^{U(1)})_*(\partial_{10}([U]))=[L_1\oplus L_{-1}]-2[1].
\end{equation}\label{417}
Using Theorem~\ref{thm:transfer}, we rewrite the left hand side as
\begin{equation}\label{418} 
(\partial_{10}\circ \nu_*)([U])=[L_1\oplus L_{-1}]-2[1],
\end{equation}
which proves that $[L_1\oplus L_{-1}]-2[1]$ belongs to $\partial_{10}(K_1(C(S^3_H)))$.
Finally, since $\nu_*$ is an isomorphism, $\nu_*([U])$ is a generator of $K_1(C(S^3_H))$, so the right hand side of 
\eqref{418} is a generator of $\partial_{10}(K_1(C(S^3_H)))$.
\end{proof}

\section{Milnor idempotents and elementary projections in $\mathcal{T}$}
\label{sec-computation-generators-exp}
\noindent
The goal of this section is to express the third generators of $K_0(C(\C P^2_q))$ and $K_0(C(\C P^2_H))$
(see Lemma~\ref{LemGenTilda}
and Corollary~\ref{MainCorSuff}, respectively) as $K_0$-classes of projections
inside the C*-algebras $C(\C P^2_q)$ and $C(\C P^2_H)$, respectively.
Note that this means that these generators are in the positive cones of their $K_0$-groups,
\begin{alignat*}{2}
[\tilde L_1\oplus \tilde L_{-1}]-2[1] &=[\hspace{1pt}\widetilde{\mathrm{p}}\hspace{1pt}] , \hspace{1.5cm} &&
\widetilde{\mathrm{p}}\in C(\C P^2_q),
\\{}
[L_1\oplus L_{-1}]-2[1] &=[\hspace{1pt}\mathrm{p}\hspace{1pt}] , && \mathrm{p}\in C(\C P^2_H).
\end{alignat*}
Note also that this is impossible in the classical setting because there are no non-trivial vector bundles of rank $0$.
Here, in the noncommutative setting, because both Peter--Weyl comodule algebras $\mathcal{P}_{U(1)}(C(S^5_q))$ and $\mathcal{P}_{U(1)}(C(S^5_H))$ have characters, the rank of any associated projective module equals the dimension of its defining corepresentation by Proposition~\ref{prop:21}.
Therefore, the only associated projective module of rank $0$ is the zero module, so the aforementioned generators cannot be given as classes of associated projective modules.
Better still, using the index pairing, we will also show that the projective modules given by $1-\widetilde{\mathrm{p}}$ and $1-\mathrm{p}$ cannot be modules associated to $\mathcal{P}_{U(1)}(C(S^5_q))$ and $\mathcal{P}_{U(1)}(C(S^5_H))$, respectively.

\subsection{The Vaksman--Soibelman case}

It was proved in \cite[Theorem 2.4]{masuda1990noncommutative} that
the group $K_1(C(SU_q(2)))\cong\Z$ is generated by the class of the unitary
\begin{equation}\label{eq:w1}
\widetilde{w}:=1+(\gamma-1)\big(1-\alpha^*(\alpha\alpha^*)^{-1}\alpha \big) .
\end{equation}
Computing the index pairing of $[\widetilde{w}]$ with the same Fredholm module that was used in \cite{dhhmw12} to compute the index pairing with the 
class of the fundamental representation $U$ of $SU_q(2)$,
one can show that $[U]=[\widetilde{w}]$.
Here we give an alternative proof of this equality by constructing an explicit homotopy between the two unitaries.

\begin{lemma}\label{prop:51}
$[U]=[\widetilde{w}]\in K_1(C(SU_q(2)))$.
\end{lemma}
\begin{proof}
For $0\leq t<1$, consider the bounded operators $\alpha_t,\gamma_t$ on $\ell^2(\N\times\Z)$ defined by
\begin{align*}
\alpha_t\ket{n,k} &:=
\begin{cases}
\sqrt{1-t^{2n}}\ket{n-1,k-1} , & \text{if }t\neq 0 , \\
\ket{n-1,k-1} , & \text{if }t=0,n\neq 0, \\
0 , & \text{if }t=0,n=0 ,
\end{cases} \\
\gamma_t\ket{n,k} &:=
\begin{cases}
t^n\ket{n,k-1} , & \text{if }t\neq 0, \\
\delta_{n,0}\ket{0,k-1} , & \text{if }t=0 .
\end{cases}
\end{align*}
Here $\ket{n,k}$, with $n\in\N$ and $k\in\Z$, are the vectors of the standard orthonormal basis of $\ell^2(\N\times\Z)$.
Note that the maps $t\mapsto\alpha_t$ and $t\mapsto \gamma_t$ are norm-continuous.

We use the injective *-homorphism $p_R:C(SU_q(2))\to\mathcal{T}\otimes C(S^1)$, and the standard identification of $u^*$ with the negative shift on $\ell^2(\Z)$ in the Fourier basis, to identify $C(SU_q(2))$ with a concrete C*-algebra of bounded operators on $\ell^2(\N\times\Z)$.
Note that this representation is unitarily equivalent to the well-known faithful representation of $C(SU_q(2))$ given in 
\cite[Corollary 2.3]{masuda1990noncommutative}.
Note also that  $\alpha_t$ and $\gamma_t$ are in the image of this representation for all $0\leq t<1$.

Now, let
\[
U_t:=\begin{pmatrix} \alpha_t & -t\gamma_t^* \\ \gamma_t & \alpha_t^* \end{pmatrix}.
\]
The fundamental representation of $SU_q(2)$ is given by $U=U_q$. The map
\[
[0,q] \ni t\longmapsto U_t
\]
is a continuous path of unitary matrices with entries in $C(SU_q(2))$. 
In particular, $U_q$ is homotopic to the unitary
\[
U_0=\begin{pmatrix} \alpha_0 & 0 \\ \gamma_0 & \alpha_0^* \end{pmatrix} .
\]
Consider now the norm-continuous map
\[
[0,1]\ni t\longmapsto V_t:=\begin{pmatrix} \sqrt{1-t}\,\alpha_0 & 
-\sqrt{t} \\ \gamma_0+\sqrt{t}\,\alpha_0^*\alpha_0 & \sqrt{1-t}\,\alpha_0^* \end{pmatrix} .
\]
One easily checks that $V_t$ is unitary for all values of $t$. This gives a homotopy between $V_0=U_0$ and
\[
V_1=\begin{pmatrix} 0 & -1 \\ \gamma_0+\alpha_0^*\alpha_0 & 0 \end{pmatrix} .
\]
Finally, observe that, by \cite[Theorem 4.2.9]{wegge1993k}, the $K_1$-class of $V_1$ is the same as the $K_1$-class of the scalar unitary $\gamma_0+\alpha_0^*\alpha_0$. By an explicit calculation using the standard basis of $\ell^2(\N\times\Z)$, one checks that $\gamma_0+\alpha_0^*\alpha_0=\widetilde{w}$.
\end{proof}

The above lemma allows us to compute the $K_0$-class $[\tilde{L}_1\oplus \tilde{L}_{-1}]-2[1]$ as the Milnor connecting homomorphism
applied to [$\widetilde{w}$]. Remarkably, this calculation agrees already on the level of projections (before taking the $K_0$-classes) with
an alternative calculation afforded by computing projections for $\tilde{L}_1$ and $\tilde{L}_{-1}$ using a strong connection (see Section~2.1) on the
principal comodule algebra $\mathcal{P}_{U(1)}(C(S^5_q))$ much as it was done in \cite{hw10} in the simpler case of $\mathcal{P}_{U(1)}(C(SU_q(2)))$.
Herein, however, the strong-connection calculation, although elementary, is very lengthy, so the homotopy shortcut of Lemma~\ref{prop:51} that
allows us to use the Milnor connection homomorphism is truly beneficial.
To use this shorcut, we rewrite the unitary 
$\widetilde{w}=\gamma_0+\alpha_0^*\alpha_0$
 as
\begin{equation}\label{eq:w2}
\widetilde{w}=\big( 1, ss^* \otimes 1 + (1-ss^*) \otimes u^* \big)
\end{equation}
taking advantage of the identification of $C(SU_q(2))$ with the pullback C*-algebra of \eqref{eq:pullsuq2}.

\begin{lemma}\label{prop:52}
One has $\partial_{10}([\widetilde{w}])=[\hspace{1pt}\widetilde{\mathrm{p}}\hspace{1pt}]\in K_0(C(\C P^2_q))$, where
\begin{equation}\label{eq:wtp}
\widetilde{\mathrm{p}}:=\big(0,(1-ss^*)\otimes (1-ss^*)\big) .
\end{equation}
\end{lemma}

\begin{proof}
We compute $\partial_{10}([\widetilde{w}])$ using the pullback 
diagram~\eqref{eq:CP2qpullback}. 
To this end, recall that $C(B^4_q)$ is the C*-subalgebra of $\mathcal{T}\otimes\mathcal{T}$ generated by 
$x:=s\otimes 1$ and $y:=(1-ss^*)\otimes s$.
Therefore, the elements
\[
c:=xx^*+y^*=ss^*\otimes 1+(1-ss^*)\otimes s^*
\qquad\text{and}\qquad d:=c^*
\]
belong to $C(B^4_q)$.
Now, using \eqref{eq:varpi} and \eqref{eq:w2}, one checks that
$\varpi(c)=\widetilde{w}$.
Furthermore, since
\begin{align*}
cd &=ss^*\otimes 1+(1-ss^*)\otimes 1=1\otimes 1 ,
\\
dc &=ss^*\otimes 1+(1-ss^*)\otimes ss^* ,
\end{align*}
the idempotent \eqref{MilnProjEq} becomes
\[
p_{\widetilde{w}}=\left(
\begin{array}{cc}
(1, 1) & (0,0)\\
(0, 0) & (0, 1-dc)
\end{array}
\right)= 
\left(
\begin{array}{cc}
(1, 1) & (0,0)\\
(0, 0) & \big(0,(1-ss^*)\otimes (1-ss^*)\big)
\end{array}\right) .
\]
Therefore, $\partial_{10}([\widetilde{w}])=[p_{\widetilde{w}}]-[1]=[\hspace{1pt}\widetilde{\mathrm{p}}\hspace{1pt}]$.
\end{proof}

Combining Lemma~\ref{prop:52}, Lemma~\ref{prop:51} and Lemma~\ref{LemGenTilda}, we obtain:
\begin{thm}\label{ptilde}
$[\tilde{L}_1\oplus \tilde{L}_{-1}]-2[1]=[\hspace{1pt}\widetilde{\mathrm{p}}\hspace{1pt}]\in K_0(C(\C P^2_q))$.
\end{thm}
\begin{cor}\label{cor:54}
The element
$[\tilde{L}_1\oplus \tilde{L}_{-1}]-2[1]$ is in the positive cone of $K_0(C(\C P^2_q))$.
\end{cor}

Note that, in the classical case, every vector bundle on $\C P^2$ is stably isomorphic to a direct sum of line bundles associated to 
the principal $U(1)$-bundle $S^5\to\C P^2$.
In particular, the positive cone of $K^0(\C P^2)$ consists only of classes of vector bundles associated to the above principal bundle.
Also, since there are no non-zero vector bundles with rank $0$,
the class $[\mathrm{L}_1\oplus\mathrm{L}_{-1}]-2[1]$ is not in the positive cone of $K_0(C(\C P^2))$ (nor in the negative cone).
However, in the quantum case, $[\tilde{L}_1\oplus \tilde{L}_{-1}]-2[1]=[\hspace{1pt}\widetilde{\mathrm{p}}\hspace{1pt}]$ is in the positive cone despite 
$C(\C P^2_q)\widetilde{\mathrm{p}}$ not being associated to the compact quantum principal $U(1)$-bundle $S^5_q\to\C P^2_q$ (see the next 
proposition). Better still, we also show that the same property is shared by the complementary projection $1-\widetilde{\mathrm{p}}$.

\begin{prop}\label{nonass}
Let $\widetilde{\mathrm{p}}$ be the projection in \eqref{eq:wtp}. Then the modules
 $C(\C P^2_q)\widetilde{\mathrm{p}}$ and $C(\C P^2_q)(1-\widetilde{\mathrm{p}})$ are not associated to $\mathcal{P}_{U(1)}(C(S^5_q))$.
\end{prop}
\begin{proof}
Recall first that there is a homomorphism of abelian groups
\[
(\tau^0,\tau^1):K_0(C(S^2_q))\longrightarrow\Z\oplus\Z
\]
assigning to each $K_0$-class its ``rank'' and its ``winding number'' \cite[(1.6)]{h-pm00}.
Theorem~2.1 in \cite{h-pm00} says that $(\tau^0,\tau^1)$, applied to the modules associated via the corepresentation 
\mbox{$1\mapsto u^n\otimes 1$}, is given by $(1,-n)$.
Now, it follows from \cite[Theorem~5.1]{hnpsz} that the left \mbox{$C(\C P^2_q)$-module} $\tilde L_n$ pushed forward via the map 
$C(S^5_q)\to C(SU_q(2))$ in \eqref{eq:diagS5q} is the module associated to $C(SU_q(2))$ via the same corepresentation.
Composing $(\tau^0,\tau^1)$ with the induced map $K_0(C(\C P^2_q))\to K_0(C(S^2_q))$ between the K-groups of the fixed-point subalgebras, we obtain 
a homomorphism of abelian groups
\[
K_0(C(\C P^2_q))\longrightarrow \Z\oplus\Z
\]
assigning to each projective module its rank and winding number, which for $\tilde L_n$ is $(1,-n)$.

Next, note that $[\hspace{1pt}\widetilde{\mathrm{p}}\hspace{1pt}]\neq 0$ because it generates a free summand of $K_0(C(\C P^2_q))$ by 
Lemma~\ref{prop:52} and Lemma~\ref{LemGenTilda}.
Therefore, $[1-\widetilde{\mathrm{p}}]=[1]-[\hspace{1pt}\widetilde{\mathrm{p}}\hspace{1pt}]\neq [1]$.
Suppose now that 
\[
C(\C P^2_q)\widetilde{\mathrm{p}}\cong\mathcal{P}_{U(1)}(C(S^5_q))\mathbin{\Box}^{\mathcal{O}(U(1))}V
\]
for some corepresentation $V$. Then it follows from Corollary~\ref{cor:54} that its $K_0$-class is $[\tilde L_1]+[\tilde L_{-1}]-2[1]$, which has rank $0$. 
Combining this with Proposition~\ref{prop:21}, we infer that $V=0$, which contradicts $[\hspace{1pt}\widetilde{\mathrm{p}}\hspace{1pt}]\neq 0$.

Suppose finally that
\[
C(\C P^2_q)(1-\widetilde{\mathrm{p}})\cong\mathcal{P}_{U(1)}(C(S^5_q))\mathbin{\Box}^{\mathcal{O}(U(1))}V'
\]
for some corepresentation $V'$.
Since this module has rank $1$, Proposition~\ref{prop:21} implies that
$\dim V'=1$. Therefore,
$C(\C P^2_q)(1-\widetilde{\mathrm{p}})\cong\tilde{L}_n$ for some $n\in\Z$.
On the other hand, it follows from
\[
[\tilde L_n]=[1-\widetilde{\mathrm{p}}]=3[1]-[\tilde L_1]-[\tilde L_{-1}]
\]
that the winding number of $[\tilde L_n]$ is \mbox{$-n=0-(-1)-(+1)=0$}.
This yields the desired contradiction as $[1-\widetilde{\mathrm{p}}]\neq [1]$.
\end{proof}

\begin{rmk}
Recall that $\mathcal{T}$ and $C(S^1)$ can be understood as graph C*-algebras in a standard way, with the natural $U(1)$-actions becoming the gauge actions. Thus, the pullback presentation of $C(SU_q(2))$ given in \eqref{eq:pullsuq2} forces us to choose a particular $U(1)$-equivariant identification of $C(SU_q(2))$ with the graph C*-algebra
$C^*(L_3)$ in~\cite{hong02}.
Under this identification, the unitary $\widetilde{w}$ in \eqref{eq:w1} becomes the inverse of the unitary in \cite[Proposition~2.3]{arici2022}
used to compute the Milnor connecting homomorphism, which turns out to be given by a vertex projection.
\end{rmk}

\subsection{The multipushout case}

The goal of this section is to transfer the results of the preceeding section to the multipushout quantum complex projective plane $\C P^2_H$.
For starters, we have:
\begin{lemma}
The class of the unitary
\begin{equation}\label{eq:wH}
w:=\big( 1\otimes 1, ss^* \otimes 1 + (1-ss^*) \otimes u^* \big)
\end{equation}
generates $K_1(C(S^3_H))$.
\end{lemma}
\begin{proof}
Note that the *-homomorphism $\nu$ in \eqref{eq:nu226} maps the unitary $\widetilde{w}$ in \eqref{eq:w2} to the unitary $w$ in \eqref{eq:wH}.
Now, since $[\widetilde{w}]$ generates $K_1(C(SU_q(2)))$, and
\[
\nu_*:K_1(C(SU_q(2)))\longrightarrow K_1(C(S^3_H))
\]
is an isomorphism by Lemma~\ref{lemma:nu43},
we conclude that 
\begin{equation}\label{ww}
[w]=\nu_*([\widetilde{w}])
\end{equation}
 generates $K_1(C(S^3_H))$.
\end{proof}

We are now ready  to compute the third generator of~$K_0(C(\C P^2_H))$:
\begin{lemma}\label{511}
The Milnor connecting homomorphism $\partial_{10}$ maps $[w]\in K_1(C(S^3_H))$ to $[\mathrm{p}]\in K_0(C(\C P^2_H))$, where
\begin{equation}\label{eq:wpnot}
\mathrm{p}:=\big(0,(1-ss^*)\otimes (1-ss^*)\big) \in C(\C P^2_H).
\end{equation}
\end{lemma}
\begin{proof}
First, recall that $\partial_{10}([\widetilde{w}])=[\hspace{1pt}\widetilde{\mathrm{p}}\hspace{1pt}]$, where $\widetilde{\mathrm{p}}$ 
is the projection \eqref{eq:wtp}.
Now, applying Theorem~\ref{thm:transfer} to the commutative diagram \eqref{ProjectiveMaps}, we deduce that
\[
\partial_{10}([w])=
(\partial_{10}\circ\nu_*)[\widetilde{w}]=
((f^{U(1)})_*\circ\partial_{10})([\widetilde{w}])=
(f^{U(1)})_*([\hspace{1pt}\widetilde{\mathrm{p}}\hspace{1pt}]) .
\]
Finally, an immediate calculation gives 
\begin{equation}\label{tilde}
(f^{U(1)})_*([\hspace{1pt}\widetilde{\mathrm{p}}\hspace{1pt}]) =[\mathrm{p}],
\end{equation}
 where $\mathrm{p}$ is the projection in~\eqref{eq:wpnot}.
\end{proof}

We are now ready to prove:
\begin{thm}\label{pnotilde}
$[{L}_1\oplus {L}_{-1}]-2[1]=[\hspace{1pt}{\mathrm{p}}\hspace{1pt}]\in K_0(C(\C P^2_H))$.
\end{thm}
\begin{proof}
It follows from \eqref{418}, Lemma~\ref{prop:51}, \eqref{ww} and Lemma~\ref{511} that
\[
[{L}_1\oplus {L}_{-1}]-2[1]=(\partial_{10}\circ \nu_*)([U])=(\partial_{10}\circ \nu_*)([\widetilde{w}])=\partial_{10}([w])
=[\hspace{1pt}{\mathrm{p}}\hspace{1pt}].
\]
\end{proof}
\begin{cor}
The element
$[L_1\oplus L_{-1}]-2[1]$ is in the positive cone of $K_0(C(\C P^2_H))$.
\end{cor}

We end by claiming the non-association result analogous to Proposition~\ref{nonass}:
\begin{prop}
Let $\mathrm{p}$ be the projection in \eqref{eq:wpnot}. Then the modules
 $C(\C P^2_H)\mathrm{p}$ and $C(\C P^2_H)(1-\mathrm{p})$ are not associated to $\mathcal{P}_{U(1)}(C(S^5_H))$.
\end{prop}
\begin{proof} 
Suppose first that 
\[
C(\C P^2_H)\mathrm{p}\cong\mathcal{P}_{U(1)}(C(S^5_H))\mathbin{\Box}^{\mathcal{O}(U(1))}V
\]
for some left $\mathcal{O}(U(1))$-comodule~$V$. Then it follows from Theorem~\ref{pnotilde} that its $K_0$-class is $[ L_1]+[ L_{-1}]-2[1]$, which has rank $0$. 
Combining this with Proposition~\ref{prop:21}, we infer that $V=0$, which contradicts the fact that $[\hspace{1pt}\mathrm{p}\hspace{1pt}]$ generates a free summand of $K_0(C(\C P^2_H))$ by Lemma~\ref{511}.

Suppose  now that $C(\C P^2_H)(1-\mathrm{p})$ is associated to $\mathcal{P}_{U(1)}(C(S^5_H))$ via a left $\mathcal{O}(U(1))$-comodule~$V'$. Using the same comodule, we define  the finitely generated projective module
\[
P_{V'}:=\mathcal{P}_{U(1)}(C(S^5_q))\mathbin{\Box}^{\mathcal{O}(U(1))}V'.
\] 
Now, it follows from \cite[Theorem~5.1]{hnpsz} that $f^{U(1)}$ maps the left \mbox{$C(\C P^2_q)$-module} $P_{V'}$ to $C(\C P^2_H)(1-\mathrm{p})$. Furthermore, since $f^{U(1)}$ induces an isomorphism $(f^{U(1)})_*$ in K-theory by Theorem \ref{KIsomMainThm}, using \eqref{tilde}, we obtain
\[
[1-\widetilde{\mathrm{p}}]=(f^{U(1)})_*^{-1}([1-\mathrm{p}])=[P_{V'}].
\]
Finally, since all modules in the same $K_0$-class have the same rank, we can reason as in Proposition~\ref{nonass} 
to conclude that $P_{V'}\cong \tilde{L}_n$ for some $n\in\mathbb{Z}$, and obtain 
 the desired contradiction by computing the winding number using the index pairing.
\end{proof}

\section*{Acknowledgements}
\noindent
This work is part of the project ``Applications of graph algebras and higher-rank graph algebras
in noncommutative geometry'' partially supported by NCN grant UMO-2021/41/B /ST1/03387.
We are very grateful to Elizabeth Gillaspy and Carla Farsi
for their contributions in early stages of this work, whose preliminary version was part of the  EU project ``Quantum Dynamics''.

\end{document}